%% file: revision-2-point-5.tex
\documentclass[12pt]{article}
\usepackage{fullpage,comment,authblk,stackrel}
\usepackage[small]{caption}
\usepackage{tikz-cd}
\usepackage{soul}
\usepackage[normalem]{ulem}

\input{commands}

\hypersetup{
  colorlinks   = true, 
  urlcolor     = blue, 
  linkcolor    = blue, 
  citecolor   = blue 
}

\newcommand{\N}			{{{\mathbb{N}}}}

\renewcommand{\Vec}		{{{\mathsf{Vec}}}}
\newcommand{\rank}		{{{\mathsf{rk\,}}}}
\newcommand{\CPmod}      	{{\mathsf{CPMod}}}

\theoremstyle{plain}
\newtheorem{thm}{Theorem}[section]
\newtheorem{lem}[thm]{Lemma}
\newtheorem{prop}[thm]{Proposition}
\newtheorem{cor}[thm]{Corollary}

\theoremstyle{definition}
\newtheorem{defn}[thm]{Definition}
\newtheorem{ex}[thm]{Example}

\theoremstyle{remark}
\newtheorem{rmk}[thm]{Remark}

\title{Galois Connections in Persistent Homology}
\author[1]{Aziz Burak G\"ulen\footnote{aziz.burak.gulen@gmail.com}}
\author[1]{Alexander McCleary\footnote{Corresponding Author, alex.mccleary@gmail.com}}

\affil[1]{Department of Mathematics, Ohio State University}
\date{}

\begin{document}

\maketitle

\begin{abstract}
We present a new language for persistent homology in terms of Galois connections.
This language has two main advantages over traditional approaches.
First, it simplifies and unifies central concepts such as interleavings and matchings.
Second, it provides access to Rota's Galois connection theorem---a powerful tool with many potential applications in applied topology.
To illustrate this, we use Rota's Galois connection theorem to give a simple proof of the bottleneck stability theorem.
Finally, we use this language to establish relationships between various notions of multiparameter persistence diagrams.
\end{abstract}

\section{Introduction}

High-dimensional data sets pose unique challenges in data analysis.
Statistics provides a wealth of quantitative tools for tackling high-dimensional data.
Topological data analysis (TDA), on the other hand, offers \emph{qualitative} invariants for understanding high-dimensional data sets~\cite{carlsson2009topology}.
Persistent homology is one of the main tools used in TDA.
It takes as input a nested sequence of spaces and outputs an invariant that captures where holes were born and died in the sequence of spaces.
This invariant is called the persistence diagram or barcode.

\begin{figure}
    \centering
    \includegraphics[scale=1]{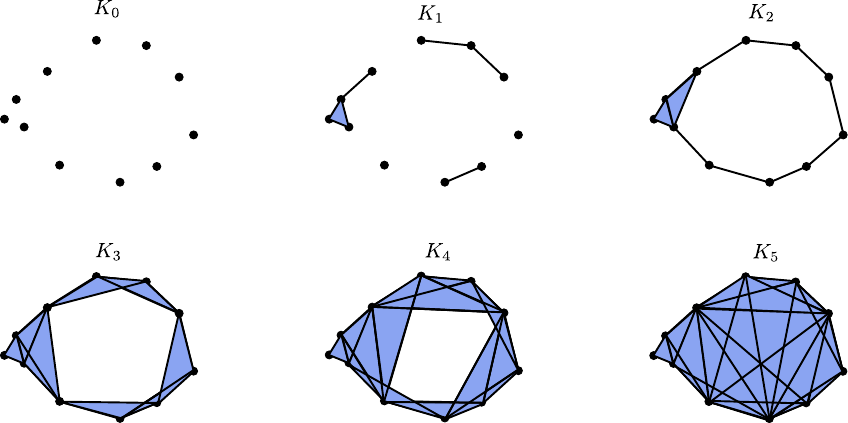}
    \caption{A filtration indexed by the totally ordered set $\{0 \leq 1 \leq 2 \leq 3 \leq 4 \leq 5\}$. There is 1-dimensional homology born at 2 and dying at 5.}
    \label{fig: filtration}
\end{figure}

These nested sequences of spaces are known as \emph{filtrations}.
In applications, they usually arise from data in the form of finite metric spaces.
Given a finite metric space $X$, there are numerous ways of obtaining a filtration.
These include Vietoris-Rips complexes, \v Cech complexes, and witness complexes.
Each of these methods produces a simplicial complex $K_t$ for each $t\in [0,\infty)$ with inclusions $K_t \mono K_s$ for all $t \leq s$; see Figure \ref{fig: filtration}.

The fundamental idea of persistent homology is to refrain from choosing a single $K_t$ to represent the geometry of $X$ and instead consider the entire filtration at once.
The persistence diagram accomplishes this by encoding where and how the homology of the filtration changes.
This turns out to be surprisingly simple to implement and relatively quick to compute; current algorithms are matrix multiplication time and often faster with real-world data \cite{ripser}.

The first step towards defining persistence diagrams is to apply homology with coefficients in a fixed field $\kk$ to the filtration $\{K_t\}$.
This produces a sequence of vector spaces $H_d(K_t)$.
The inclusions $K_t \mono K_s$ induce linear maps from $H_d(K_t)$ to $H_d(K_s)$ for each $t\leq s \in [0,\infty)$.
This data is known as a \emph{persistence module} and can be more succinctly defined as a functor from the poset category $[0,\infty)$ to the category of finite-dimensional $\kk$-vector spaces $\Vec$.

To convert a persistence module into a persistence diagram, we use an intermediate step called \emph{nullity function}; see Definition~\ref{defn: nullity function}
This function takes as input intervals $(a,b) \subseteq [0,\infty)$ and captures the number of homological features that were born at or before $a$ and die at or before $b$.
We now employ a tool from order theory to convert the nullity function into a persistence diagram: M\"obius inversion.

M\"obius inversion can be thought of as a combinatorial analog of a derivative.
Just as a derivative captures how a function changes, the M\"obius inversion captures where, and how, a function defined on a poset changes.
Nullity functions change where homology classes are born and die.
This is precisely the data isolated by the persistence diagram.

Persistence diagrams of totally-ordered filtrations are remarkably well-behaved.
They are a complete invariant \cite{zomorodian2004computing}, meaning no information is lost when passing from a persistence module to a persistence diagram.
Additionally, persistence diagrams of totally-ordered filtrations are a stable invariant~\cite{CSEdH}, meaning that changes to a persistence module induce proportional changes to its persistence diagram.

\subsection{Historical Context}

Here we present a brief, and by no means complete, overview of some historical developments leading to our results.
Persistent homology arose independently in the works of Robins~\cite{vanessa}; Frosini and Landi \cite{landi-pseudodistances}; and Edelsbrunner, Letscher, and Zomorodian~\cite{Edelsbrunner2002}.

One of the most significant developments since the inception of persistent homology is the bottleneck stability theorem \cite{CSEdH}.
This theorem implies that if a data set is perturbed by $\ee$, the resulting persistence diagram will be perturbed by at most $\ee$.
This result is crucial for justifying applications of persistent homology as real-world data always has some amount of noise present.

In 2009, Carlsson and Zomorodian began the study of multiparameter persistence with \emph{The Theory of Multidimensional Persistence}~\cite{Carlsson2009}.
In multiparameter persistence, the totally ordered set $[0,\infty)$ is replaced with the partially ordered set $[0,\infty)^n$.
Filtrations indexed by $[0,\infty)^n$ arise frequently and have a wealth of potential applications \cite{botnan-intro-multiparameter}.
Unfortunately, Carlsson and Zomorodian showed that there is no complete, discrete invariant for multiparameter persistent homology unlike the 1-parameter setting.
This setback has not slowed down the search for rich invariants of multiparameter persistence modules and it remains a very active area of study today.

Current approaches to multiparameter persistence mostly revolve around partial invariants.
In their 2009 paper, Carlsson and Zomorodian proposed the rank function as an alternative to persistence diagrams.
Cerri et al.\ encoded the same information in an invariant known as the fibered barcode, obtained by restricting a multiparameter persistence module to every affine line of non-negative slope and taking the usual 1-parameter persistence diagram \cite{landi,landi2}.
Lesnick and Wright developed a software package, RIVET, for computing fibered barcodes \cite{RIVET}.

The algebraic-combinatorial approach to persistent homology was pioneered by Patel in~\cite{Patel2018} where persistence diagrams were constructed by a M\"obius inversion.
This has allowed for typical methods of constructing persistence diagrams to be extended to multiparameter persistence modules \cite{kim2018generalized}.
Since \cite{Patel2018}, M\"obius inversion has been used extensively throughout persistent homology \cite{bottleneck, kim2018generalized, betthauser2019graded, botnan2022signed, fasy2022persistent}.
A major development with the algebraic-combinatorial approach to persistence was the introduction of the birth-death function by Henselman-Petrusek~\cite{saecular}.
While birth-death functions do not explicitly appear in \cite{saecular}, the main idea is present there.
Birth-death functions for filtrations were explicitly written up in \cite{editdistance}.

\subsection{Our Contributions}

Galois connections offer a new framework for the study of persistent homology.
This framework integrates seamlessly with the main existing algebraic-combinatorial tool used in persistent homology: M\"obius inversion.
In Section \ref{sec: RGCT}, we examine the relationship between Galois connections and the M\"obius inversion operator.
In particular, we prove a new formulation of Rota's Galois connection theorem (RGCT); see Theorem~\ref{thm:RGCT}.
Traditional statements of RGCT~\cite{rota64, greenemobius} involve the M\"obius function while our version states RGCT directly through the M\"obius inversion operator.
Our approach both simplifies the theorem's statement and facilitates its application in persistent homology. 

In Section \ref{sec: pmods-and-pdgms}, we introduce a category of persistence modules $\Pmod$ and a category of persistence diagrams $\Dgm$.
The category $\Pmod$ is closely related to the category of filtrations defined over finite lattices from \cite{editdistance}.
The objects in $\Pmod$ are functors from posets to the category of finite-dimensional vector spaces and the morphisms are given by pullbacks along right adjoints.
The category $\Dgm$ has persistence diagrams as objects (which we define as integer-valued functions over posets of intervals; see Definition \ref{def: diagrams}) where the morphisms are given by pushforwards along order-preserving maps.
We then define a functor from the full subcategory of persistence modules satisfying a finiteness condition to the category of persistence diagrams.
This functor is constructed by first defining an intermediary function for each persistence module, the nullity function (Definition \ref{defn: nullity function}), and then applying a M\"obius inversion.
RGCT, which implies that M\"obius inversion takes pullbacks along right adjoints to pushforwards along left adjoints, then immediately implies the assignment of a persistence diagram to each persistence module is functorial; see Proposition \ref{prop:functoriality}.

In Section \ref{sec: metrics}, we interpret interleavings between persistence modules and matchings between persistence diagrams in our language.
We introduce a new notion of interleavings which we call Galois interleavings.
While these Galois interleavings are significantly more flexible than the traditional interleavings for persistence modules over $[0,\infty)$, in Proposition \ref{prop: old interleaving} we show that, in this setting, Galois interleavings induce the same metric.
Both Galois interleavings and matchings are simply spans in the categories $\Pmod$ and $\Dgm$ respectively.
That is, if $\Mfunc$ and $\Nfunc$ are persistence modules then a Galois interleaving between $\Mfunc$ and $\Nfunc$ is a persistence module $\Gamma$ with morphisms $\Mfunc \leftarrow \Gamma \rightarrow \Nfunc$.
Likewise, if $w_0$ and $w_1$ are persistence diagrams then a matching between them is a nonnegative persistence diagram $\nu$ with morphisms $w_0 \leftarrow \nu \rightarrow w_1$. 
In Section \ref{sec: stability}, we exploit the flexibility provided by Galois interleavings together with functoriality to give a proof of the Bottleneck Stability Theorem (Theorem \ref{thm:bottleneck stability}).

In Section \ref{sec: multiparameter}, we extend the birth-death function of~{\cite{editdistance}} from filtrations to persistence modules by using free covers.
We then utilize RGCT to establish previously unknown relationships between various notions of multiparameter persistence diagrams.
In particular, we prove that the persistence diagrams obtained from nullity functions and birth-death functions are equivalent; see Proposition~\ref{prop:ker and bd}.
We also show that rank function persistence diagrams and birth-death function (or nullity function) persistence diagrams are closely related  (Proposition \ref{prop:relation diagrams}) and provide a formula for the fibered barcode in terms of the persistence diagram using RGCT (Proposition \ref{prop: fibered linear time}). 

\subsection{Related Work}

This paper builds off of \cite{editdistance} and many concepts presented here are refinements of ideas present there.
The work of Bubenik et al.\ on interleavings \cite{bubenik2017interleaving, bubenik-higher-interpolation} has been a particularly strong influence.
Our version of (Galois) interleaving is closely related to the version present in \cite{bubenik2017interleaving}.
Additionally, Bubenik et al.\ have also studied persistence modules over preordered sets \cite{bubenikgpm}, although they take a different approach than us.
The finiteness condition used in this paper extends the notion of constructibility used in \cite{Patel2018} and is more restrictive than finitely presentable or the tameness condition in \cite{miller}.
Our finiteness condition is similar to the notion of tame functor in \cite{Scolamiero2017}.

\subsection{Acknowledgements}

We thank Amit Patel for many helpful conversations and providing valuable feedback.
In particular, the method of defining birth-death functions for persistence modules used here was suggested by Patel.
We are grateful to Ezra Miller for his detailed feedback and suggestions to improve the paper.
Additionally, we are thankful to the anonymous reviewers for the many improvements they suggested.
We would also like to thank Facundo M\'emoli.
The language of pushforwards and pullbacks used throughout this paper stemmed from conversations with M\'emoli.

This work is partially supported by NSF-IIS-1901360, NSF-DMS-1547357, NSF-CCF-1839356, NSF-CCF-1740761, and NSF-CCF-1839358.

\section{Preliminaries}

Here we present some background on partially ordered sets and M\"obius inversion.
For a more thorough reference, see \cite{greenemobius}.

\subsection{Partially Ordered Sets}

\begin{defn}
    A partially ordered set, or \define{poset}, is a set $P$ with a relation $\leq$ satisfying
    \begin{itemize}
        \item Reflexivity: $a\leq a$ for any $a \in P$.
        \item Antisymmetry: if $a\leq b$ and $b \leq a$ then $a=b$ for any $a,b \in P$.
        \item Transitivity: if $a\leq b$ and $b \leq c$ then $a \leq c$ for any $a,b,c \in P$.
    \end{itemize}
\end{defn}

\begin{ex}\label{ex:poset}
    A recurring example we use throughout this paper is the poset ${P = \{ a, b, c, d\}}$ with the partial order
        $a \leq c$, $b \leq c$, and $c \leq d$.
    See Figure \ref{fig:poset} for the Hasse diagram of $P$.
\end{ex}

\begin{figure}
    \centering
    \begin{tikzcd}
         &d \\
        &c\arrow[u, no head]\\
        a\arrow[ur, no head] & &b\ar[ul, no head]
    \end{tikzcd}
    \caption{
        The Hasse diagram of the poset $P$ from Example \ref{ex:poset}.
    }
    \label{fig:poset}
\end{figure}

\begin{defn} \label{def:downsets}
    For any poset $P$, a subposet $A \subseteq P$ is an \define{upset} if $a \in A$ and $b \in P$ with~$a \leq b$ implies that $b \in A$.
    For any element $a \in P$ the principal upset of a is the set
    \[
        \uparrow a := \{ b \in P \,\big|\, a\leq b \}.
    \]
    Dually, a subposet $B\subseteq P$ is a \define{downset} if $b \in B$ and $a \in P$ with~$a \leq b$ implies $a \in B$.
    The principal downset of an element $b \in P$ is the set
    \[
        \downarrow b := \{ a \in P \mid a\leq b \}.
    \]
    The set of all non-empty downsets in $P$, ordered by inclusion, forms a poset denoted~$\down(P)$.
    There is an embedding of $P$ into $\down(P)$ via $a \mapsto \downarrow a$.
\end{defn}

The persistence diagram of a persistence module is a function capturing information that persists over intervals in a poset.
Therefore, in order to define the persistence diagram of a module, we first define posets of intervals.

\begin{defn}\label{def: posets-of-ints}
    For any poset $P$, there are two \define{posets of intervals} that appear in applied topology, $\bar P$ and $\hat P$. Both posets of intervals have the same underlying set:
    \[
        \big\{ (a,b) \in P \times  P \,|\, a\leq b \big\}.
    \]
    The poset $\bar P$ has the product order $(a,b) \leq (c,d)$ if $a\leq c$ and $b \leq d$ while the poset $\hat P$ has the reverse containment order $(a,b) \supseteq (c,d)$ if $a\leq c$ and $d \leq b$.
    The \define{diagonal} of $\bar P$ (or~$\hat P$) is the subset $\Delta_P := \{(a,a) \mid a \in P\}$ and the \define{off-diagonal intervals} are $\bar{P}^\circ := \bar P - \Delta_P$.
\end{defn}

Note that an order-preserving map $f: P \to Q$ induces an order-preserving map $\bar f : \bar P \to \bar Q$ by $\bar f(a,b) = \big(f(a), f(b) \big)$ but, in general, fails to induce an order-preserving map from $\hat P$ to~$\hat Q$.

\begin{defn}
    Let $P$, $Q_0$, and $Q_1$ be posets with order-preserving maps
    \begin{center}
        \begin{tikzcd}
            Q_0\ar[dr, "f"]   &   &Q_1\ar[dl, "g"']\\
                &P.
        \end{tikzcd}
    \end{center}
    The \define{fibered product} of $Q_0$ and $Q_1$ by $P$ is the poset $Q_0 \times_P Q_1$ with underlying set
    $$\{ (a,b) \in Q_0 \times Q_1 \mid f(a) = g(b) \}$$
    and the inherited product order.
    It is easily checked that this poset is the limit of the above diagram in the category of posets with order-preserving maps.
\end{defn}

It will occasionally be useful for us to loosen the definition of a poset by dropping the antisymmetry condition.

\begin{defn}
    A \define{preordered set} is a set $P$ with a relation $\leq$ satisfying reflexivity and transitivity.
\end{defn}

One of the main reasons for introducing preordered sets is that they can serve as a convenient domain for defining certain functors. Specifically, we use preordered sets in the proof of the triangle inequality for the Galois interleaving distance in Section~\ref{subsec: triangle ineq}.

Preordered sets have a disadvantage that prevents us from using them more often: M\"obius inversions are not necessarily unique.
This issue can be overcome in certain circumstances, in fact, M\"obius inversion can be defined even more generally for certain categories \cite{leroux-mobius, lawvere-mobius}. 
However this is not necessary for this paper and so we leave it out.

\subsection{Galois Connections}

Galois connections play a fundamental role throughout this paper.
Here we provide the necessary background on Galois connections.


\begin{defn}
    A \define{Galois connection} between two posets $P$ and $Q$ consists of order-preserving maps $f: P \leftrightarrows Q :g$ satisfying $f(p)\leq q$ if and only if $p \leq g(q)$ for any $p\in P$ and $q\in Q$.
    We refer to $f$ as the left adjoint and $g$ as the right adjoint of the Galois connection.
    We will occasionally use the notation $(f,g)$ for a Galois connection when the domains of $f$ and $g$ are clear from context.
\end{defn}

\begin{rmk}\label{rmk: unique left-right adjoint via formula}
    If a poset map $f:P \to Q$ has a right adjoint, then it is uniquely defined and given by the formula $g (q) = \max \{ p \in P \mid f(p)\leq q \}$.
    Similarly, if $g: Q \to P$ has a left adjoint it is unique and given by $f(p) = \min \{ q \in Q \mid p \leq g(q) \}$ \cite[Section 5]{greenemobius}.
\end{rmk}

The following lemma is an immediate consequence of Remark \ref{rmk: unique left-right adjoint via formula}.

\begin{lem} \label{lem:finite-totally-ordered-adjoints}
    If $P$ and $Q$ are finite, totally ordered sets and $f: P \to Q$ is an order-preserving map, then $f$ has a right adjoint if and only if $f$ sends the minimum element of $P$ to the minimum element of $Q$.
    Similarly, $f$ has a left adjoint if and only if $f$ sends the maximum element of $P$ to the maximum element of $Q$.
\end{lem}

\begin{rmk}
    There is another historically used definition of Galois connections.
    An antitone Galois connection consists of order-reversing maps $f: P \leftrightarrows Q: g$ satisfying $q \leq f(p)$ if and only if $p \leq g(q)$.
    Antitone Galois connections are, in a sense, equivalent to Galois connections as follows.
    Let $P^\op$ be the opposite poset of $P$: the poset with the same underlying set and the opposite order.
    Then $f: P \leftrightarrows Q :g$ is a Galois connection if and only if $f: P^\op \leftrightarrows Q :g$ is an antitone Galois connection.
    We will not make use of antitone Galois connections in this paper.
\end{rmk}

\begin{ex}
    The inclusion of $\Z$ into $\R$ has both a left adjoint: the ceiling function; and a right adjoint: the floor function.
\end{ex}

\begin{defn}
    A \define{Galois insertion} is a Galois connection $f: P \leftrightarrows Q :g$  satisfying $f \circ g = \id_{Q}$.
\end{defn}

\begin{lem}\label{lem: inf and def}
    Let $f : P \leftrightarrows Q : g$ be a Galois connection.
    Then for all $p \in P$ and $q \in Q$,
    \[
        p \leq g \big( f (p) \big) \text{ and } f \big( g(q) \big) \leq q .
    \]

\end{lem}
\begin{proof}
    The Galois connection condition implies that
    \begin{align*}
        f(p) \leq f(p)  &\iff p \leq g \big( f(p) \big)\\
        g(q) \leq g(q)  &\iff f \big( g(q) \big) \leq q.
    \end{align*}
    The inequalities on the left hold for all $p\in P$ and $q \in Q$ and thus, the inequalities on the right hold for all $p \in P$ and $q \in Q$.
\end{proof}

\begin{lem}\label{lem: idempotent}
    If $f : P \leftrightarrows Q : g$ is a Galois connection, then
    \begin{align*}
        f \circ g \circ f &= f\\
        g \circ f \circ g &= g.
    \end{align*}
    Moreover, the maps $f \circ g$ and $g \circ f$ are idempotent.
\end{lem}
\begin{proof}
    For any $p \in P$, Lemma \ref{lem: inf and def} and the fact that $f$ is order-preserving give
    \[
        p \leq g \big( f(p) \big) \implies f(p) \leq f \Big( g \big( f(p) \big) \Big).
    \]
    The reverse inequality follows from
    \[
        g \big( f(p) \big) \leq g \big( f(p) \big) \iff f \Big( g \big( f(p) \big) \Big) \leq f(p)
    \]
    which holds for all $p \in P$ and thus, $f \circ g \circ f = f$.
    The proof that $g \circ f \circ g = g$ follows analogously.
    The fact that~$f \circ g$ and $g \circ f$ are idempotent follows immediately from the fact that $f \circ g \circ f = f$.
\end{proof}

The proofs of the following propositions are straightforward and follow directly from the definitions.

\begin{prop}
    If $f:P \leftrightarrows Q :g$ is a Galois connection (insertion), then $(f,g)$ extends to a Galois connection (insertion) between the posets of intervals $\bar f: \bar P \leftrightarrows \bar Q :\bar g$ defined componentwise as $\bar f: (a,b) \mapsto (f(a),f(b))$ and $\bar g: (c,d) \mapsto (g(c),g(d))$.
\end{prop}

\begin{prop}
    If $f : P \leftrightarrows Q :g$ and $h : Q \leftrightarrows R: \ell$ are Galois connections (insertions), then $h \circ f : P \leftrightarrows R: g \circ \ell$ is a Galois connection (insertion).
\end{prop}

\begin{defn}\label{def: pushout-poset}
    Let $(Q_0, \leq_0)$, $(Q_1,\leq_1)$, and $(P,\leq_P)$ be posets with Galois insertions $(f_0,g_0)$ and $(f_1,g_1)$
    \begin{center}
        \begin{tikzcd}
            Q_0\ar[dr,bend right,"f_0"']     &   &Q_1\ar[dl, bend left, "f_1"]\\
                &P\ar[ul, bend right, "g_0"]\ar[ur, bend left, "g_1"']
        \end{tikzcd}
    \end{center}
    The \define{gluing} of $Q_0$ and $Q_1$ along $P$ is the preordered set $Q_0 \sqcup_P Q_1$ with the underlying set~${ \big( Q_0 \times \{0\} \big) \cup \big( Q_1 \times \{1\} \big)}$ and order $(a,t) \preceq (b,s)$ if
    \begin{itemize}
        \item $t=s$ and $a \leq_t b$ or
        \item $t\neq s$ and $g_s \big( f_t(a) \big) \leq_s b$.
    \end{itemize}
\end{defn}

It is not immediately obvious that the relation $\preceq$ is transitive in the case where $t \neq s$ and $(a,t) \preceq (b,s) \preceq (c,t)$.
Indeed, the Galois insertion condition is necessary for this to hold.
If $(a,t) \preceq (b,s) \preceq (c,t)$ then $g_s \big( f_t(a) \big) \leq_s b$ and $g_t \big( f_s(b) \big) \leq_t c$.
Therefore $g_t \Big( f_s \big( g_s( f_t(a) ) \big) \Big) \leq_t c$ and, since $f_s \circ g_s = \id$, this reduces to $g_t \big( f_t(a) \big) \leq_t c$.
By Lemma~\ref{lem: inf and def}, we have $a \leq_t g_t \big( f_t(a) \big)$ and so $a \leq_t c$ as desired.

\begin{rmk}
    The definition of Galois connections (and insertions) applies as is to the more general setting of preordered sets rather than posets.
    In this paper, we make use of a Galois connection between preordered sets in only one instance: the proof of the triangle inequality for the Galois interleaving distance (and in Proposition \ref{prop: concatenation insertion}, which is utilized in the proof of the triangle inequality).
\end{rmk}

\begin{prop}\label{prop: concatenation insertion}
    Let $Q_0$, $Q_1$, and $P$ be posets with Galois insertions $(f_0,g_0)$ and $(f_1,g_1)$ as in Definition \ref{def: pushout-poset}.
    Then for $s \in \{0,1\}$ there are Galois insertions $\pi_s : Q_0 \sqcup_P Q_1 \leftrightarrows Q_s : \iota_s$ given by $\iota_s(a) = (a,s)$ and
    \[
        \pi_s (a,t) =
        \begin{cases}
            a   &\text{if } t=s\\
            g_s \big( f_t (a) \big) &\text{otherwise.}
        \end{cases}
    \]
    Moreover, the following diagram commutes.
    \[
    \begin{tikzcd}
           &Q_0 \sqcup_P Q_1\ar[dl, "\pi_0"']\ar[dr, "\pi_1"]   \\
        Q_0\ar[dr, "f_0"'] &   &Q_1\ar[dl, "f_1"]    \\
            &P  
    \end{tikzcd}
    \]

\end{prop}
The proof is trivial.
The following concept will play an important role in the finiteness conditions we assume throughout the paper.

\begin{defn}
    A \define{co-closure operator} on a poset $P$ is an order-presering map ${c: P \to P}$ that is deflationary: meaning $c(p) \leq p$ for all $p \in P$; and idempotent: meaning $c \circ c =c$.
\end{defn}

\begin{prop}\label{prop: coclosure}
    Let $P$ be a poset and $c:P \to P$ be an order-preserving map.
    Consider the canonical factorization of $c$ through its image
        \[
    \begin{tikzcd}
        P\ar[rr, "c"]\ar[dr, two heads, "\pi"']   &   &P\\
            &\image(c).\ar[ur, hookrightarrow, "\iota"']
    \end{tikzcd}
    \]
    Then $c$ is a co-closure operator if and only if $(\iota, \pi)$ form a Galois connection.
\end{prop}
\begin{proof}
    Suppose that $(\iota, \pi)$ form a Galois connection.
    Then Lemma \ref{lem: inf and def} implies that $c = \iota \circ \pi$ is deflationary and Lemma \ref{lem: idempotent} implies that $c$ is idempotent.
    Thus, $c$ is a co-closure operator.

    Conversely, suppose that $c$ is a co-closure operator.
    First observe that $\iota$, being the inclusion of the subposet $\image(c)$ into $P$, satisfies
    \[\iota(x) \leq \iota(y) \iff x \leq y\]
    for all $x,y \in \image(c)$.
    Now for any $x \in \image(c)$ and $p \in P$, we must show
    \[
        \iota(x) \leq p \iff x \leq \pi(p).
    \]
    Since $c$ is idempotent, if $x \in \image(c)$ then $c(\iota(x)) = \iota(x)$.
    Therefore,
    \begin{align*}
        \iota(x) \leq p &\implies c \big( \iota(x) \big) \leq c(p)\\
        &\implies \iota(x) \leq c(p)\\
        &\implies \iota(x) \leq \iota \big( \pi(p) \big)\\
        &\implies x \leq \pi(p).
    \end{align*}
    Now, using the fact that $c$ is deflationary ($c(p) \leq p$ for all $p \in P$), the converse implication follows from,
    \begin{align*}
        x \leq \pi(p)   &\implies   \iota(x) \leq \iota \big( \pi(p) \big)\\
        &\implies \iota(x) \leq c(p)\\
        &\implies \iota(x) \leq p.
    \end{align*}
    Therefore, $(\iota, \pi)$ form a Galois connection.
\end{proof}

\subsection{M\"obius Inversion}

The M\"obius inversion is the core tool at the heart of the algebraic-combinatorial approach to persistence developed in \cite{Patel2018}.
It can be thought of as a combinatorial version of the derivative: for a function $m:P \to \Z$, its M\"obius inversion $\partial_P m :P \to \Z$ captures where, and how, $m$ changes.
Here we cover only the necessary background for this paper and refer the reader to \cite{rota64} for a more thorough overview of M\"obius inversion. In this paper we consider Möbius inversion in infinite settings, for which we need to resort to the notion of Möbius invertible functions; see \cite[Definition 3.1]{CKM} and \cite{clause2024}. The following Definition \ref{defn: mobius invertability} is based on, and a natural generalization of, \cite[Definition 3.1]{CKM}.

\begin{defn}\label{defn: mobius invertability}
    Let $P$ be a poset and $m: P \to \Z$ be a function.
    A \define{M\"obius inversion} of $m$ is a function $m' : P \to \Z$ satisfying
    \begin{equation}\label{eq: MI formula}
        m(b) = \sum_{a \leq b} m'(a) \text{ for all } b\in P,
    \end{equation}
    where the right-hand side contains only finitely many nonzero summands.
    We refer to this equation as the M\"obius inversion formula. A function $m : P \to \Z$ is called \define{M\"obius invertible} if a M\"obius inversion of $m$ exists.
\end{defn}



\begin{prop}\label{prop: mi is unique}
    Let $P$ be a poset and $m : P \to \Z$ be a M\"obius invertible function. Then, $m$ has a unique M\"obius inverse.
\end{prop}
\begin{proof}
    The proof is omitted, as the same argument used in \cite[Proposition 3.2]{CKM} works in this case as well. While their original result applies specifically to the generalized rank invariant, the underlying logic extends directly to any integer-valued function on an arbitrary poset.
\end{proof}

\begin{rmk}
    By the proposition above, every M\"obius invertible function $m : P \to \mathbb{Z}$ admits a unique M\"obius inversion.  
    From now on, we denote this inversion by $\partial_P m : P \to \mathbb{Z}$, and refer to it as \emph{the} M\"obius inversion of $m$.
\end{rmk}

\begin{rmk}
    When the poset $P$ is finite, every function $m: P \to \Z$ has a M\"obius inversion~\cite{rota64}. Moreover, the M\"obius inversion can be computed via matrix multiplication as described in Appendix~\ref{app: mobius}. For a fixed poset $P$, we view $\partial_P$ as a linear operator on the set of all M\"obius invertible functions from $P$ to $\Z$.
When the underlying poset $P$ is clear, we will drop the subscript and simply write $\partial$ for the M\"obius inversion operator on $P$.    
\end{rmk}

\begin{rmk}
    The relationship between a function $m$ and its M\"obius inversion $\partial_P m$ is analogous to the relationship between the cumulative distribution function (CDF) and the probability density function (PDF) of a probability measure.
Explicitly, if $\alpha$ is a probability measure on $\R$ with CDF $F_\alpha$ then its PDF, if it exists, is the unique function $f_\alpha$ satisfying
\[
    F_\alpha(x) = \int_{t \leq x} f_\alpha(t) \, dt.
\]
\end{rmk}

\section{Rota's Galois Connection Theorem}\label{sec: RGCT}

The following version of Rota's Galois connection theorem is the main tool used in this paper.
This statement of the theorem was developed as a generalization of Proposition 6.6 in \cite{editdistance}.
An equivalent version appears independently in \cite{sanchezhopf}.
Although it differs from classical statements of Rota's Galois connection theorem\cite{rota64, greenemobius}, we refer to Theorem~\ref{thm:RGCT} as Rota's Galois connection theorem, as justified by Remark~\ref{rmk:rgct}.

\begin{defn}
    Let $P$ and $Q$ be posets and $f: P \to Q$ be a function.
    For any $m : P \to \Z$, the \define{pushforward} of $m$ along $f$ is the function $f_\sharp m : Q \to \Z$ given by
    \[
        f_\sharp m : q \mapsto \sum_{p \in f^{-1}(q)} m(p).
    \]
\end{defn}

\begin{defn}
    Let $P$ and $Q$ be posets and $f: P \to Q$ be a function.
    For any $n: Q \to \Z$, the \define{pullback} of $n$ along $f$ is the function $f^\sharp n : P \to \Z$ defined by $f^\sharp n = n \circ f$.
\end{defn}

\begin{thm}[Rota's Galois Connection Theorem]\label{thm:RGCT}
    Let $P$ and $Q$ be posets, and ${f:P \leftrightarrows Q :g}$ be a Galois connection.
    For any function $m: P \to \Z$, if $m$ is M\"obius invertible, then $g^\sharp (m)$ is M\"obius invertible and
    \[
        \partial_Q \big(g^\sharp (m) \big) = f_\sharp ( \partial_P m ).
    \]
    
\end{thm}
\begin{proof}
    Suppose $m: P \to \Z$ is M\"obius invertible.
    To show that $f_\sharp (\partial_P m)$ is the M\"obius inversion of $g^\sharp (m)$ we show that
    \[
        m \big( g(b) \big) = \sum_{a \leq b} f_\sharp \big(\partial_P m \big)(a)
    \]
    for all $b \in Q$.
    This follows from
    \begin{align*}
        \sum_{a \leq b} \big( f_\sharp \partial_P m \big) (a) &= \sum_{a\leq b} \sum_{p \in f^{-1}(a)} \partial_P m(p)\\
        &=\sum_{\substack{p\in P \\ f(p) \leq b}} \partial_P m(p)\\
        &=\sum_{p\leq g(b)} \partial_P m(p)\\
        &= m(g(b)).
    \end{align*}
\end{proof}

\subsection{Constructible Functions}

The functions arising in applied topology are typically not defined over finite posets so we need to be careful to ensure they have M\"obius inversions.
Here we present a finiteness condition that guarantees M\"obius invertibility.

\begin{defn}\label{def: constructible functions}
    Let $P$ be any poset, $m: P \to \Z$ a function, and $S\subseteq P$ a finite subposet.
    We say that $m$ is \define{$S$-constructible} if there is a co-closure operator $c: P \to P$ with image $S$ such that $m = m \circ c$.
    We say that $m$ is \define{constructible} if there is some finite poset $S$ such that $m$ is $S$-constructible.
\end{defn}

\begin{prop}\label{prop: constructible MI}
    Let $m:P \to \Z$ be an $S$-constructible function as in Definition \ref{def: constructible functions}.
    Then $\partial m$ exists and is given by $\partial m = \iota_\sharp (\partial m|_S )$ where $\iota$ is the inclusion of $S$ into $P$.
\end{prop}
\begin{proof}
    Since $m$ is $S$-constructible, there exists a co-closure operator $c: P \to P$ with image $S$ and satisfying $m = m \circ c$.
    By Proposition \ref{prop: coclosure}, the inclusion $\iota: S \hookrightarrow P$ has a right adjoint $\pi: P \to S$ with $c = \iota \circ \pi$.
    Now since $S$ is finite, $m|_S = m \circ \iota$ is M\"obius invertible.
    Therefore,
    \[
        m = m\circ c = m \circ \iota \circ \pi  = m|_S \circ \pi = \pi^\sharp (m|_S)
    \]
    and so, by RGCT, $\partial m = \iota_\sharp (\partial m|_S)$.
\end{proof}

\section{Persistence Modules and Persistence Diagrams}\label{sec: pmods-and-pdgms}

In this section, we introduce categories of persistence modules and persistence diagrams.
The main result in this section is Proposition \ref{prop:functoriality} which shows that the assignment of a persistence diagram to each constructible persistence module is functorial.

\subsection{Persistence Modules}\label{sec: pmods}

Persistence modules are the typical starting point for studying persistent homology from a theoretical point of view.
They arise in practice by applying homology to a filtration.
In this section we define persistence modules and morphisms of persistence modules.

Let $\Vec$ be the category of finite dimensional vector spaces over a fixed field $\kk$.

\begin{defn}\label{def: cat of mods}
    The \define{category of persistence modules} is the category $\Pmod$ with
    \begin{itemize}
        \item Objects: functors $\Mfunc: P \to \Vec$ for any poset $P$. We call these objects \define{persistence modules}.
        \item Morphisms: If $\Mfunc: P \to \Vec$ and $\Nfunc: Q \to \Vec$ are persistence modules, a morphism from $\Mfunc$ to $\Nfunc$ is a Galois insertion $f: P \leftrightarrows Q :g$ satisfying $\Mfunc \circ g \cong \Nfunc$.
    \end{itemize}
\end{defn}

Note that there is another frequently used notion of morphism between persistence modules: natural transformations \cite{induced_matching}.
There are two limitations with natural transformations as morphisms.
The first is that natural transformations only exist between persistence modules defined over the same poset.
The second is that natural transformations between persistence modules do not induce a functorial relationship between persistence diagrams~\cite[Proposition 5.10]{induced_matching}.
On the other hand, the category of persistence modules with natural transformations as morphisms enjoys some nice properties (in particular, it is an abelian category).
To avoid confusion with the morphisms in Definition \ref{def: cat of mods}, we use the notation $\Rightarrow$ to represent a natural transformation and $\to$ to represent a morphism.

\begin{ex}
    For any poset $P$ and any $a \in P$ we define the persistence module ${\kk^{\uparrow a} : P \to \Vec}$ by
    \[
        \kk^{\uparrow a} (b):= 
        \begin{cases} 
            \kk &\text{ if }  a \leq b\\
            0   &\text{otherwise}
        \end{cases}
    \]
    and $\kk^{\uparrow a}(b\leq c) = \id_\kk$ for all $a\leq b \leq c$.
\end{ex}

\begin{defn}\label{def: const fun}
    Let $P$ be a poset and let $S \subseteq P$ be a finite subposet.
    A persistence module $\Mfunc: P \to \Vec$ is $S$-\define{constructible} if there is a co-closure operator $c: P \to P$ with image $S$ such that $\Mfunc \circ c \cong \Mfunc$.
    We say that $\Mfunc$ is \define{constructible} if $\Mfunc$ is $S$-constructible for some finite subposet $S$.
\end{defn}

\begin{rmk}
    Note that, in general, there is no unique set $S$ such that a persistence module $\Mfunc:P \to \Vec$ is $S$-constructible. However, in Appendix~\ref{appendix: constructible}, we show that there is always a unique smallest such set.
\end{rmk}

\begin{defn}
    The \define{category of constructible persistence modules} is the full subcategory $\CPmod$ of $\Pmod$ with constructible persistence modules as objects.
\end{defn}

\begin{rmk}\label{rmk: no-const-on-R}
    Note that there are no constructible persistence modules defined on $\R$.
    This is due to the fact that there are no co-closure operators with finite images on $\R$.
    There are, however, plenty of constructible functors defined on the posets $[-\infty,\infty]$, $[-\infty,\infty)$ $[0,\infty]$, and $[0,\infty)$.
\end{rmk}

\begin{ex}\label{ex:module}
    Figure \ref{fig:module} shows a persistence module $\Mfunc : P \to \Vec$ where $P$ is the poset from Example \ref{ex:poset} and the map $\theta$ is given by $\theta: (x,y) \mapsto x+y$.
\end{ex}

\begin{figure}
    \[
    \begin{tikzcd}
        &0\\
        &\kk \ar[u]\\
        \kk^2 \ar[ur, "\theta"] & & \kk \ar[ul, "\id"']
    \end{tikzcd}
    \]
    \caption{
        A persistence module over $P$; see Example \ref{ex:module}.
        The map $\theta$ is defined as $\theta: (x,y) \mapsto x+y$.
    }
    \label{fig:module}
\end{figure}

\subsection{Persistence Diagrams}\label{sec: pd}

While persistence modules are the usual starting point for theoretical foundations of persistent homology, persistence diagrams are the typical end point.
Persistence diagrams offer a ``topological summary'' of filtrations or persistence modules, capturing where homological features are born and die.
In this section we construct persistence diagrams out of constructible persistence modules via nullity functions, see Definitions~\ref{defn: nullity function} and~\ref{def: dgm-of-a-mod}.
We organize persistence diagrams into a category and show that the map taking a constructible persistence module to its persistence diagram is functorial.


Recall from Definition \ref{def: posets-of-ints} that for a poset $P$, its poset of intervals is $\bar P$, the diagonal is $\Delta_P = \{(a,a) \mid a \in P \} \subseteq \bar P$, and the off-diagonal intervals are $\bar{P}^\circ = \bar P - \Delta_P$.
Moreover, any order-preserving map $f: P \to Q$ extends to a function $\bar f: \bar P \to \bar Q$ by applying $f$ component-wise.

\begin{defn}\label{def: diagrams}
    A \define{persistence diagram} is a finitely supported (not necessarily order-preserving) function $w : \bar P \to \Z$ for some poset $P$.
\end{defn}

\begin{defn}
    The \define{category of persistence diagrams}, denoted $\Dgm$, has
    \begin{itemize}
        \item Objects: persistence diagrams.
        \item Morphisms: If $w_0: \bar P \to \Z$ and $w_1: \bar Q \to \Z$ are persistence diagrams, a morphism from $w_0$ to $w_1$ is an order-preserving function $f: P \to Q$ such that $\bar f_\sharp (w_0)|_{\bar{Q}^\circ} = w_1 |_{\bar{Q}^\circ}$.
    \end{itemize}
\end{defn}

Note that we do not require the map $f$ defining a morphism to be a left adjoint of a Galois insertion.
This condition will often times be satisfied but it is not necessary to assume it.
Unpacking the definition of morphisms in $\Dgm$, we get that a morphism from $w_0: \bar P \to \Z$ to $w_1: \bar Q \to \Z$ is an order-preserving map $f: P \to Q$ such that, for all $J \in \bar Q^\circ$,
\[
    w_1(J) = \sum_{I \in \bar f^{-1}(J)} w_0(I).
\]
Intuitively, a morphism from $w_0$ to $w_1$ is a map $f$ such that $\bar f$ pushes-forward $w_0$ onto $w_1$ everywhere except for possibly the diagonal.
Two persistence diagrams $w_0: \bar P \to \Z$ and~${w_1: \bar Q \to \Z}$ are isomorphic in $\Dgm$ if there is an isomorphism of posets $f: P \to Q$ such that $w_0 (a,b) = w_1 \big( f(a),f(b) \big)$ for any $a < b \in P$.

\begin{defn}\label{defn: nullity function}
    For any persistence module $\Mfunc: P \to \Vec$, its \define{nullity function} is the function $\mathsf{null}_\Mfunc : \bar P \to \Z$ given by
    \[
        \mathsf{null}_\Mfunc (a,b) = \dim \big( \ker \Mfunc(a \leq b) \big).
    \]
\end{defn}

If $\Mfunc$ is a constructible persistence module, then $\nul_\Mfunc$ is a constructible function and therefore M\"obius invertible.
To see this, suppose that $\Mfunc$ is $S$-constructible for some finite subposet $S \subseteq P$.
Then there exists a co-closure operator $c: P \to P$ with $\image(c) = S$ and $\Mfunc \circ c \cong \Mfunc$.
Now, for any $a \leq b \in P$, we have
\[
    \nul_\Mfunc(a,b) = \dim \big( \ker \Mfunc(a \leq b) \big) = \dim \Big( \ker
        \Mfunc \big( c(a) \leq c(b) \big) \Big) = \nul_\Mfunc \circ \bar c (a,b).
\]
Therefore, $\nul_\Mfunc$ is $\bar S = \image(\bar c)$ constructible.

\begin{defn}\label{def: dgm-of-a-mod}
    Let $\Mfunc: P \to \Vec$ be a constructible persistence module.
    A \define{persistence diagram of $\Mfunc$} is any persistence diagram $w: \bar P \to \Z$ isomorphic to $\partial \mathsf{null}_\Mfunc : \bar P \to \Z$ under the identity map.
    That is, $w$ is a persistence diagram of $\Mfunc$ if $w(a,b) = \partial \mathsf{null}_\Mfunc (a,b)$ for any $a<b \in P$.
\end{defn}

\begin{rmk}[Connection to the Classical Persistence Diagrams]\label{rmk: relation to classical pd}
     The term \emph{persistence diagram} has a long history in topological data analysis, dating back over two decades~\cite{Edelsbrunner2002}. Our Definition~\ref{def: dgm-of-a-mod} should be understood as a generalization of this classical concept from the setting of filtrations indexed by a totally ordered set to that of modules indexed by a general poset.

     To be precise, when a persistence module arises from a standard filtration of a topological space over a totally ordered set, the persistence diagram produced by our definition coincides with the classical persistence diagram. A full justification of this fact requires the machinery of birth-death functions, which we develop in Section~\ref{sec: multiparameter}. We explore this equivalence in Remark~\ref{rmk: establish equivalence}
\end{rmk}

\begin{prop}[Functoriality]\label{prop:functoriality}
    There is a functor $\PD: \CPmod \to \Dgm$ assigning to each persistence module $\Mfunc$ the persistence diagram $\partial \mathsf{null}_\Mfunc$.
\end{prop}
\begin{proof}
    For a constructible persistence module $\Mfunc$, define $\PD(\Mfunc)$ as $\PD(\Mfunc) = \partial \mathsf{null}_\Mfunc$.
    Let $\Mfunc: P \to \Vec$ and $\Nfunc: Q \to \Vec$ be two constructible persistence modules with a morphism $f: P \leftrightarrows Q: g$ from $\Mfunc$ to $\Nfunc$.
    We will show that $f$ is a morphism from $\partial \mathsf{null}_\Mfunc$ to $\partial \mathsf{null}_\Nfunc$.
    By definition, $\Mfunc \circ g \cong \Nfunc$ and so $\mathsf{null}_{\Mfunc \circ g} (a,b) = \mathsf{null}_{\Nfunc} (a,b)$ for any $(a,b) \in \bar Q$.
    However $\mathsf{null}_{\Mfunc \circ g}(a,b) = \mathsf{null}_{\Mfunc} \big( g(a), g(b) \big)$ and therefore $\mathsf{null}_{\Mfunc} \circ \bar g = \mathsf{null}_\Nfunc$.
    Now RGCT (Theorem~\ref{thm:RGCT}) implies $\bar f_\sharp (\partial \mathsf{null}_\Mfunc) = \partial \mathsf{null}_\Nfunc$.
    Therefore, $f$ is a morphism from $\PD(\Mfunc)= \partial \mathsf{null}_\Mfunc$ to $\PD(\Nfunc) = \partial \mathsf{null}_\Nfunc$.
\end{proof}

\begin{prop}[Nonnegativity]\label{prop: positivity}
    If $P = \{ p_1 < ... < p_n \}$ is a finite totally ordered set and $\Mfunc : P \to \Vec$ is a persistence module then $\partial \nul_\Mfunc (I) \geq 0$ for all $I \in \bar P^\circ$. Moreover, the same conclusion holds for constructible functors $\Mfunc$ defined on a 
    (not necessarily finite) totally ordered set.
\end{prop}
\begin{proof}
    Suppose $I= (p_k,p_\ell) \in \bar{P}^\circ$.
    If $k=1$ then the M\"obius inversion formula reduces to $\partial \nul_\Mfunc(I) = \nul_\Mfunc(p_1,p_\ell) - \nul_\Mfunc(p_1,p_{\ell -1})$ and since ${\Mfunc(p_1 \leq p_\ell)}$ factors through ${\Mfunc(p_1 \leq p_{\ell -1})}$, we have $\nul_\Mfunc(p_1,p_\ell) \geq \nul_\Mfunc(p_1,p_{\ell -1})$.
    
    Now suppose $k > 1$.
    To simplify the notation, let $\alpha_0$, $\alpha_1$, and $\alpha_2$ be the following linear maps induced by $\Mfunc$
    \begin{center}
        \begin{tikzcd}
            \Mfunc(p_{k-1})\ar[r,"\alpha_0"]    &\Mfunc(p_k)\ar[r,"\alpha_1"]   &\Mfunc(p_{\ell-1})\ar[r, "\alpha_2"]   &\Mfunc(p_\ell).
        \end{tikzcd}
    \end{center}
    The M\"obius inversion formula on $\bar P$ breaks down as
    \begin{align*}
        \partial \nul_\Mfunc (I) &= \nul_\Mfunc(p_k,p_\ell) - \nul_\Mfunc (p_k,p_{\ell -1}) - \big(\nul_\Mfunc (p_{k-1},p_\ell)  - \nul_\Mfunc(p_{k-1},p_{\ell -1})\big)\\
        &= \dim \bigg( \frac{\ker(\alpha_2 \circ \alpha_1)}{\ker(\alpha_1)} \bigg) - \dim \bigg( \frac{\ker(\alpha_2 \circ \alpha_1 \circ \alpha_0)}{\ker(\alpha_1 \circ \alpha_0)} \bigg).
    \end{align*}
    Observe that
    \[
        \alpha_0 \big(\ker(\alpha_2 \circ \alpha_1 \circ \alpha_0) \big) \subseteq \ker(\alpha_2 \circ \alpha_1)
    \]
    and 
    \[
        \alpha_0 \big( \ker( \alpha_1 \circ \alpha_0 ) \big) \subseteq \ker(\alpha_1)
    \]
    so $\alpha_0$ induces a map 
    \[
        \varphi: \frac{\ker(\alpha_2 \circ \alpha_1 \circ \alpha_0)}{\ker(\alpha_1 \circ \alpha_0)} \to \frac{\ker(\alpha_2 \circ \alpha_1)}{\ker(\alpha_1)}.
    \]
    Now if $x \in \ker(\alpha_2 \circ \alpha_1 \circ \alpha_0)$ with $\alpha_0(x) \in \ker(\alpha_1)$ then $x \in \ker(\alpha_1 \circ \alpha_0)$ so $\varphi$ is injective.
    Therefore
    \[
        \dim \bigg( \frac{\ker(\alpha_2 \circ \alpha_1)}{\ker(\alpha_1)} \bigg) \geq \dim \bigg( \frac{\ker(\alpha_2 \circ \alpha_1 \circ \alpha_0)}{\ker(\alpha_1 \circ \alpha_0)} \bigg)
    \]
    and so $\partial \nul_\Mfunc(I) \geq 0$.

    The extension to constructible functors on an arbitrary totally ordered set follows from Proposition~\ref{prop: constructible MI}.
\end{proof}

\section{The Galois Interleaving and Bottleneck Distances}\label{sec: metrics}

The interleaving distance and the bottleneck distance are the two most commonly used metrics in persistent homology.
They are the foundation of a critical result in the field: the bottleneck stability theorem.
While it is known that the interleaving distance between two persistence modules and the bottleneck distance between their persistence diagrams are actually equal \cite{Lesnick2015}, the definitions of interleavings and matchings seem to have very little in common.
We show that interleavings and matchings are both closely related to the notion of coupling in optimal transport \cite{villani-optimal-transport}.

First, we establish some notation and a useful lemma.
\begin{defn}
    Let $(X, d_X)$ be a metric space and let $A$ be a set.
    The metric $d_X$ extends to a metric $d_\infty$ on functions from $A$ to $X$ by
    \[
        d_\infty (f,g) = \sup_{a \in A} d_X \big( f(a), g(a) \big).
    \]
\end{defn}

The metric $d_\infty$ depends on the metric space $(X,d_X)$ and the domain $A$; however, both of these parameters are typically clear from context and thus omitted from the notation.

\begin{lem}
    Let $(X, d_X)$ be a metric space and let $A$ and $B$ be sets.
    If $f: A \to B$ and~${g,h: B \to X}$ are functions then
    \[
        d_\infty( g \circ f, h \circ f) \leq d_\infty (g, h).
    \]
\end{lem}
\begin{proof}
    This follows from
    \begin{align*}
        d_\infty (g \circ f, h \circ f) &= \sup_{a \in A} d_X \Big( g\big( f(a)\big), h \big( f(a) \big) \Big)\\
        &= \sup_{b \in \image(f)} d_X \big( g(b), h(b) \big)\\
        &\leq \sup_{b \in B} d_X \big( g(b), h(b) \big)\\
        &= d_\infty (g, h).
    \end{align*}
\end{proof}

\subsection{Galois Interleavings}

Here we first present an existing definition of interleavings and then introduce a new definition of interleavings, building off of Bubenik et al. \cite{bubenik2017interleaving}, which we call \emph{Galois interleavings}.
We show that Galois interleavings induce the same metric as the classical definition for persistence modules defined on $[0,\infty)$ and highlight a deep, and previously unknown, connection between Galois interleavings and couplings in optimal transport.

\begin{defn}[Interleavings \cite{bubenik2017interleaving}]\label{def: peter-interleavings}
    For any $\ee > 0$, let $[0,\infty) \times_{\ee} \{0,1\}$ be the poset with underlying set $[0, \infty) \times \{0,1\}$ and order $(a,t) \leq (b,s)$ if $a + \ee |t-s| \leq b$.
    An $\ee$-\define{interleaving} is a functor $\Gamma: [0,\infty) \times_\ee \{0,1\} \to \Vec$ such that the following diagram commutes up to natural isomorphism
    \begin{center}
        \begin{tikzcd}
                &{[0,\infty)} \times_\ee \{0,1\}\ar[dd,"\Gamma"]\\
            {[0,\infty)}\ar[ur,"\iota_0",hook]\ar[dr,"\Mfunc"']  &   &{[0,\infty)}\ar[ul,"\iota_1"',hook]\ar[dl,"\Nfunc"]\\
                &\Vec.
        \end{tikzcd}
    \end{center}
\end{defn}

\begin{defn}[Galois Interleavings]\label{def: interleaving}
    Let $P$ and $Q$ be posets embedded in some metric space $(X,d_X)$.
    A \textbf{Galois interleaving} between two persistence modules $\Mfunc: P \to \Vec$ and $\Nfunc: Q \to \Vec$ consists of a persistence module $\Gamma : R \to \Vec$ with morphisms
    \begin{center}
        \begin{tikzcd}
                &\Gamma\ar[dl]\ar[dr]\\
            \Mfunc  &   &\Nfunc.
        \end{tikzcd}
    \end{center}
\end{defn}

    We can break down the definition of a Galois interleaving to see that $\Gamma$ is a Galois interleaving if there exist Galois insertions $(f_\Mfunc,g_\Mfunc)$ and $(f_\Nfunc,g_\Nfunc)$ such that the solid arrows in the following diagram commute up to natural isomorphism
    \begin{center}
        \begin{tikzcd}
                &R\ar[dl,"f_\Mfunc"', bend right, dashed]\ar[dr, "f_\Nfunc", bend left, dashed]\ar[dd,"\Gamma"]\\
            P\ar[dr,"\Mfunc"']\ar[ur, bend right,"g_\Mfunc"]   &   &Q\ar[dl,"\Nfunc"]\ar[ul, bend left, "g_\Nfunc"']\\
                &\Vec.
        \end{tikzcd}
    \end{center}

\begin{defn}
    The cost of a Galois interleaving $\Gamma$ is
    \[
        ||\Gamma|| = d_\infty (f_\Mfunc, f_\Nfunc) = \sup_{r \in R} \, d_X \big( f_\Mfunc (r), f_\Nfunc(r) \big).
    \]
    The \textbf{Galois interleaving distance} between $\Mfunc$ and $\Nfunc$ is
    \[
        d_I^G (\Mfunc,\Nfunc) = \inf_{\Gamma} ||\Gamma||
    \]
    where the infimum is taken over all Galois interleavings between $\Mfunc$ and $\Nfunc$.
\end{defn}

\begin{rmk}
    The fact that the Galois interleaving distance satisfies the triangle inequality will be proven in Subsection~\ref{subsec: triangle ineq}.
\end{rmk}

Typically, the posets $P$ and $Q$ will both be $[0,\infty)$ for one-parameter persistent homology or $[0,\infty)^n$ for multiparameter persistent homology so the assumption that $P$ and $Q$ are embedded in some common metric space is reasonable.
We also note that the Galois interleaving distance depends implicitly on the underlying metric space $(X, d_X)$.

\begin{rmk}\label{rmk: interleavings-as-couplings}
    The notion of Galois interleaving in Definition \ref{def: interleaving} is closely related to the notion of couplings between probability measures used in optimal transport.
    A coupling between two probability measures $\alpha$ and $\beta$ on $\R$ is a probability measure $\gamma$ on ${\R \times \R}$ with~${\gamma(A,\R)=\alpha(A)}$ and $\gamma(\R,B)=\beta(B)$ for any measurable sets $A$ and $B$.
    Rephrasing couplings in terms of cumulative distribution functions (CDFs) elucidates the similarity between Galois interleavings and couplings.
    
    Let
    \[
        F_\alpha, F_\beta : [-\infty,\infty] \to [0,1]
    \]
    be the CDFs of $\alpha$ and $\beta$ respectively.
    We can equivalently say that a measure $\gamma$ on $\R \times \R$ with CDF $F_\gamma : [-\infty,\infty]^2 \to [0,1]$ is a coupling if the following diagram commutes
    \begin{center}
        \begin{tikzcd}
                &{[-\infty,\infty]}^2\ar[dd,"F_\gamma"]\\
            {[-\infty,\infty]}\ar[ur,"\iota_0"]\ar[dr,"F_\alpha"']    &   &{[-\infty,\infty]}\ar[ul,"\iota_1"']\ar[dl,"F_\beta"]\\
                &{[0,1]}
        \end{tikzcd}
    \end{center}
    where $\iota_0: x \mapsto (x,\infty)$ and $\iota_1: x \mapsto (\infty,x)$.
    The maps $\iota_0$ and $\iota_1$ turn out to be right adjoints of Galois insertions.
    Their left adjoints are given by the two projection maps.
\end{rmk}

\begin{ex}\label{ex: interleaving}
    Let $\Mfunc_0,\Mfunc_1: [0,\infty) \to \Vec$ and $\Gamma$ be the persistence modules defined as
    \begin{align*}
        \Mfunc_0 (a) &= 
        \begin{cases}
            \kk     &\text{if } 0 \leq a < 1\\
            \kk^2   &\text{if } 1 \leq a < 2\\
            0       &\text{if } 2 \leq a
        \end{cases}\\
        \Mfunc_1 (a) &=
        \begin{cases}
            \kk     &\text{if } 0 \leq a < 2\\
            0       &\text{if } 2 \leq a
        \end{cases}\\
    \end{align*}
    with all of the induced maps $\Mfunc_0(a \leq b)$ and $\Mfunc_1(a \leq b)$ being full-rank.
    The domain of $\Gamma$ is the poset $[0,\infty) \times_1 \{0,1\}$ with order $(x,t) \leq (y,s)$ if $x + |t-s| \leq y$, as in Definition \ref{def: peter-interleavings}.
    The persistence module $\Gamma$ is defined as
    \begin{align*}
        \Gamma (t,0) &\cong 
        \begin{cases}
            \kk     &\text{if } 0 \leq t < 1\\
            \kk^2   &\text{if } 1 \leq t < 2\\
            0       &\text{if } t \geq 2
        \end{cases} \\
        \Gamma(t,1) &\cong
        \begin{cases}
            \kk     &\text{if } 0\leq t < 2\\
            0       &\text{if } t \geq 2
        \end{cases}
    \end{align*}
    with induced morphisms as shown in Figure \ref{fig: interleaving}.
    The morphisms from $\Gamma$ to $\Mfunc_0$ and $\Mfunc_1$ are given by the Galois insertions $f_t: (x,s) \mapsto x + |t-s|$ and $g_t: x \mapsto (x,t)$.
    The cost of $\Gamma$ is $1$.
    Note that this Galois interleaving does not minimize the cost function.
    The optimal interleaving here has cost $\frac{1}{2}$ and can be defined on the poset $[0, \infty) \times_{\frac{1}{2}} \{0,1\}$ in a similar fashion.
    \begin{figure}
        \centering
        \includegraphics[scale=1]{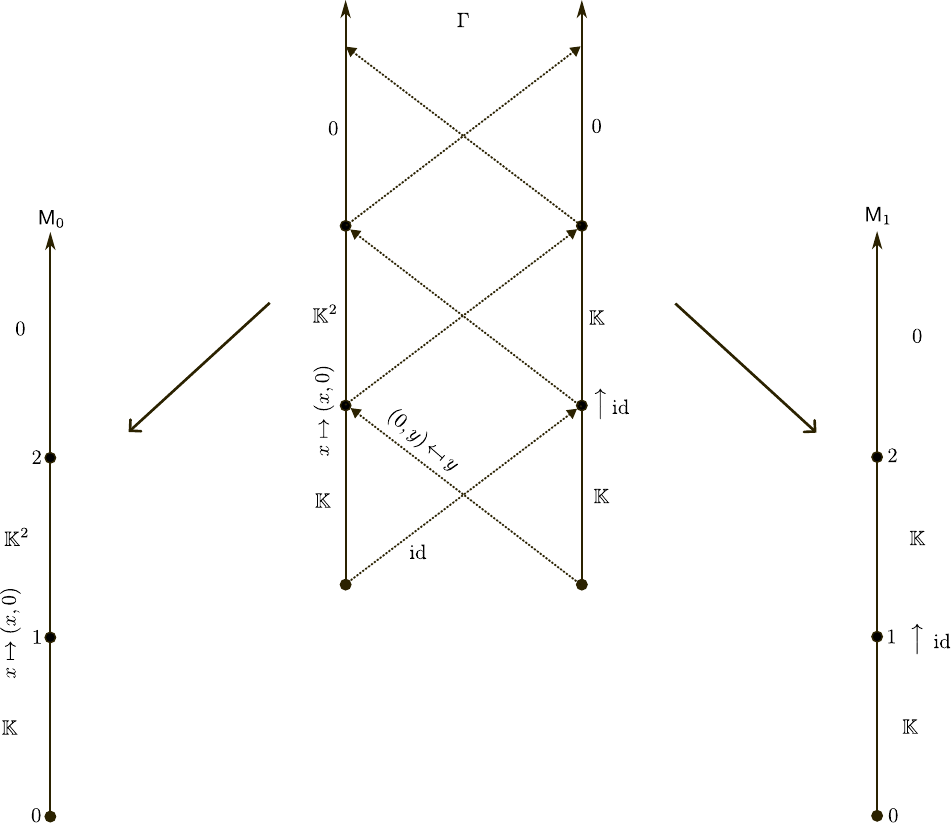}
        \caption{A Galois interleaving $\Gamma$ between the two persistence modules $\Mfunc_0$ and $\Mfunc_1$. See Example~\ref{ex: interleaving} for more details.}
        \label{fig: interleaving}
    \end{figure}
\end{ex}

The definition of Galois interleavings in \ref{def: interleaving} differs from the classical notion of interleaving in Definition \ref{def: peter-interleavings}; however, as the following proposition shows, both Galois interleavings and classical interleavings ultimately induce the same distance for persistence modules over~$[0,\infty)$.

\begin{prop}\label{prop: old interleaving}
    For any persistence modules $\Mfunc,\Nfunc: [0,\infty) \to \Vec$,
    \[
        d_I(\Mfunc, \Nfunc) = d_I^G(\Mfunc, \Nfunc).
    \]
\end{prop}
\begin{proof}
    Let $\Phi: [0,\infty) \times_\ee \{0,1\} \to \Vec$ be an $\ee$-interleaving between $\Mfunc$ and $\Nfunc$ as in Definition~\ref{def: peter-interleavings}.
    Then we have the following diagram that commutes up to natural isomorphism
    \begin{center}
        \begin{tikzcd}
                &{[0,\infty)} \times_\ee \{0,1\}\ar[dd,"\Phi"]\\
            {[0,\infty)}\ar[ur, hook,"\iota_0"]\ar[dr,"\Mfunc"']  &   &{[0,\infty)}\ar[ul, hook,"\iota_1"']\ar[dl,"\Nfunc"]\\
                &\Vec.
        \end{tikzcd}
    \end{center}
    These $\ee$-interleavings turn out to be Galois interleavings as well.
    The maps $\iota_0$ and $\iota_1$ have left adjoints $\pi_0$ and $\pi_1$ defined by $\pi_i: (x,t) \mapsto x + \ee |i-t|$ and so $\Phi$ induces a Galois interleaving.
    The cost of $\Phi$ is $d_\infty (\pi_0, \pi_1) = \ee$ which implies that $d_I^G(\Mfunc,\Nfunc) \leq d_I(\Mfunc,\Nfunc)$.

    To see the reverse inequality, let 
    \begin{center}
        \begin{tikzcd}
                &R\ar[dl,"f_0"', bend right, dashed]\ar[dr, "f_1", bend left, dashed]\ar[dd,"\Gamma"]\\
            {[0,\infty)}\ar[dr,"\Mfunc"']\ar[ur, bend right,"g_0"]   &   &{[0,\infty)}\ar[dl,"\Nfunc"]\ar[ul, bend left, "g_1"']\\
                &\Vec
        \end{tikzcd}
    \end{center}
    be a Galois interleaving between $\Mfunc$ and $\Nfunc$ with cost $\ee$.
    Define $[0,\infty) \times_R \{0,1\}$ as the preordered set with underlying set $[0,\infty) \times \{0,1\}$ and order $(x,t) \leq (y,s)$ if $g_t(x) \leq g_s(y)$.
    Let~${\Psi: [0,\infty) \times_R \{0,1\} \to \Vec}$ be the functor $\Psi(x,t):= \Gamma(g_t(x))$.
    Both $[0,\infty) \times_\ee \{0,1\}$ and $[0,\infty) \times_R \{0,1\}$ have the same underlying sets so let $\alpha: [0,\infty) \times_\ee \{0,1\} \to [0,\infty) \times_R \{0,1\}$ be the identity map.
    We proceed by showing that $\alpha$ is order-preserving.
    This will allow us to precompose the functor $\Psi$ with $\alpha$ to obtain an $\ee$-interleaving.
    
    To show that the identity map from $[0,\infty) \times_\ee \{0,1\}$ to $[0,\infty) \times_R \{0,1\}$ is order-preserving, it suffices to show that, for any $(x,t)\in [0,\infty)\times_R\{0,1\}$, $(x,t)\leq (x+\epsilon,1-t) \in [0,\infty)\times_R\{0,1\}$.
    Since $(f_{1-t}, g_{1-t})$ is a Galois connection, $g_{1-t} \circ f_{1-t}$ is inflationary (by Lemma~\ref{lem: inf and def}) and so for any $x \in [0,\infty)$, 
    \[
        g_t(x) \leq g_{1-t} \Big( f_{1-t} \big( g_t (x) \big) \Big).
    \]
    Therefore, $(x,t) \leq \Big( f_{1-t} \big( g_t(x) \big) , 1-t \Big)$ in $[0,\infty) \times_R \{0,1\}$ and so it now suffices to show that $f_{1-t} \big( g_t (x) \big) \leq x + \ee$ in $[0,\infty)$.
    Since $(f_t,g_t)$ is a Galois insertion, $x = f_t \big ( g_t(x) \big)$ and so $|x - f_{1-t} \big( g_t(x) \big)| = |f_t \big ( g_t(x) \big) - f_{1-t} \big( g_t(x) \big) | \leq d_\infty (f_t, f_{1-t}) = \ee$.

    Now we have the diagram
    \begin{center}
        \begin{tikzcd}
                &{[0,\infty)} \times_\ee \{0,1\}\ar[dd,"\Psi \circ \alpha"]\\
            {[0,\infty)}\ar[ur, hook,"\iota_0"]\ar[dr,"\Mfunc"']  &   &{[0,\infty)}\ar[ul, hook,"\iota_1"']\ar[dl,"\Nfunc"]\\
                &\Vec
        \end{tikzcd}
    \end{center}
    which commutes (up to natural isomorphism) since, by construction of $\Psi$,
    \[
        \Psi \circ \alpha \circ \iota_0 = \Gamma \circ g_0 \cong \Mfunc \hspace{4mm} \text{and} \hspace{4mm} \Psi \circ \alpha \circ \iota_1 = \Gamma \circ g_1 \cong \Nfunc. 
    \]
    Therefore, $d_I^G (\Mfunc, \Nfunc) \geq d_I (\Mfunc, \Nfunc)$.
\end{proof}

\subsection{Matchings}\label{subsec: matchings}

Just as Galois interleavings provide a means of measuring distance between persistence modules, matchings offer a way of measuring distance between persistence diagrams.
Matchings, much like Galois interleavings, are closely related to the notion of couplings in measure theory \cite{DivolMeasures}.
In this section, we assume that $P$ is a poset with a metric $d_P$. 
We extend $d_P$ to a metric $d_{\bar P}$ on the poset $\bar P$ by defining
\[
    d_{\bar P} (I, J) = \max \{ d_P( i_0, j_0), d_P( i_1, j_1) \}
\]
for any $I = (i_0, i_1)$ and $J = (j_0, j_1)$ in $\bar P$.

\begin{defn}\label{def: matching}
    A \define{matching} between two persistence diagrams $w_0, w_1 : \bar P \to \Z$ consists of a nonnegative persistence diagram $\nu : \bar{R} \to \Z$ for some poset $R$ with morphisms
    \begin{center}
        \begin{tikzcd}
                &{\nu} \ar[dl]\ar[dr]\\
            {w_0}  &   &{w_1}.
        \end{tikzcd}
    \end{center}
    \end{defn}
    Breaking this down, a matching is a nonnegative, finitely supported function~${\nu : \bar{R} \to \Z}$ with order-preserving maps $f_0 : R \to P$ and $f_1 : R \to P$ such that, for all $I \in \bar P^\circ$,
    \begin{align*}
        w_0(I) &= \sum_{J \in \bar f_0^{-1}(I)} \nu(J) \text{ and,}\\
        w_1(I) &= \sum_{J \in \bar f_1^{-1}(I)} \nu(J).
    \end{align*}

\begin{rmk}
    We define matchings between persistence diagrams only when their domains are the same.
    This assumption is typically satisfied in persistent homology as the domains will usually be $\overline{[0, \infty)}$ for one-parameter persistent homology and $\overline{[0,\infty)^n}$ for multiparameter persistent homology.
    This is more restrictive than Galois interleavings which are defined between persistence modules over potentially different posets embedded into a common metric space.
    The need for this restriction arises from subtleties with the diagonal.
\end{rmk}

\begin{defn} \label{def:bottleneck-distance}
    Let $w_0, w_1: \bar P \to \Z$ be persistence diagrams with a matching ${w_0 \xleftarrow{f_0} \nu \xrightarrow{f_1} w_1}$.
    If $\nu$ is not identically 0, we define the cost of $\nu$ as
    \[
        ||\nu|| = \max_{I \in \supp(\nu)} d_{\bar P} \big(\bar f_0(I), \bar f_1(I) \big)
    \]
    where $\text{supp}(\nu)$ is the support of $\nu$.
    If $\nu$ is identically 0, then we define $||\nu|| = 0$.
    The \define{bottleneck distance} between $w_0$ and $w_1$ is
    \[
        d_B(w_0, w_1) = \inf_{\nu} ||\nu||
    \]
    where the infimum is taken over all matchings $\nu$ between $w_0$ and $w_1$.
\end{defn}

For a matching $w_0 \xleftarrow{f_0} \nu \xrightarrow{f_1} w_1$, we make frequent use of the inequality
\[
    ||\nu|| \leq d_\infty (f_0, f_1 ).
\]

\begin{rmk} \label{rmk:matchings-equivalence}
    We have defined matchings between two persistence diagrams $w_0, w_1: \bar P \to \Z$ as having domain $\bar R$ for some poset $R$.
    However, it is often more intuitive to think of a matching between $w_0$ and $w_1$ as a finitely supported, nonnegative function $\eta: \bar P \times \bar P \to \Z$ such that, for all $I \in \bar P^\circ$,
    \begin{align}
        w_0(I) &= \sum_{J \in \bar P} \eta(I, J) \text{ and,} \label{eq:alt-matching1}\\
        w_1(I) &= \sum_{J \in \bar P} \eta(J, I). \label{eq:alt-matching2}
    \end{align}
    The cost of $\eta$ can then be defined as
    \[
        ||\eta||' = \max_{(I, J) \in \supp(\eta)} d_{\bar P}(I, J).
    \]
    This then induces a distance on persistence diagrams over $\bar P$ by taking the infimum of $||\eta||'$ over all such $\eta$.

    This alternative approach to defining a distance between persistence diagrams was used in~\cite{bottleneck} and ultimately induces the same metric as the bottleneck distance in Definition~\ref{def:bottleneck-distance}.
    To see this, we show that for any such $\eta$, there exists a matching $\nu$ between $w_0$ and $w_1$ with $||\nu|| = ||\eta||'$ and conversely, for any matching $\nu$ between $w_0$ and $w_1$, there exists such an $\eta$ with $||\eta||' = ||\nu||$.

    Suppose that $\eta: \bar P \times \bar P \to \Z$ is a finitely supported, nonnegative function satisfying Equations \eqref{eq:alt-matching1} and \eqref{eq:alt-matching2}.
    Let $\proj_0, \proj_1: P \times P \to P$ be the natural projection maps.
    We define~$\nu: \overline{P \times P} \to \Z$ as $\nu = \eta \circ (\bar \proj_0, \bar \proj_1)$.
    Then for any $I \in \bar P^\circ$,
    \begin{align*}
        \sum_{((a_0,b_0),(a_1,b_1)) \in \bar \proj_0^{-1}(I)} \nu((a_0,b_0),(a_1,b_1)) &=
        \sum_{((a_0,b_0),(a_1,b_1)) \in \bar \proj_0^{-1}(I)} \eta((a_0,a_1),(b_0, b_1))\\
        &= \sum_{(b_0,b_1) \in \bar P} \eta(I, (b_0,b_1))\\
        &= w_0(I).
    \end{align*}
    This, together with the same argument for $w_1$, shows that $w_0 \xleftarrow{\proj_0} \nu \xrightarrow{\proj_1} w_1$ is a matching.
    Observe that $(\bar \proj_0, \bar \proj_1): \overline{P \times P} \to \bar P \times \bar P$ is a bijection and so, $\supp(\eta) = (\bar \proj_0, \bar \proj_1)(\supp(\nu))$.
    Therefore, the cost of $\nu$ is
    \begin{align*}
        ||\nu|| &= \max_{ K \in \supp(\nu)} d_{\bar P} \big( \bar \proj_0(K), \bar \proj_1(K) \big)\\
        &= \max_{(I,J) \in \supp(\eta)} d_{\bar P} (I, J)\\
        &= ||\eta||'.
    \end{align*}

    Now suppose that $w_0 \xleftarrow{f_0} \nu \xrightarrow{f_1} w_1$ is a matching.
    Define $\eta: \bar P \times \bar P \to \Z$ as 
    \[
        \eta = (\bar f_0, \bar f_1)_\sharp (\nu).
    \]
    That $\eta$ is nonnegative follows immediately from $\nu$ being nonnegative.
    Since $\nu$ is nonnegative and $\eta$ is a pushforward of $\nu$,
    \begin{equation} \label{eq:alt-matching3}
        \supp(\eta) = (\bar f_0, \bar f_1) \big(\supp(\nu) \big).
    \end{equation}
    Thus, $\eta$ is finitely supported.
    For any $I \in \bar P^\circ$, we have
    \begin{align*}
        \sum_{J \in \bar P} \eta(I,J) &= \sum_{J \in \bar P} \, \sum_{K \in (\bar f_0, \bar f_1)^{-1}(I, J)} \nu(K)\\
        &= \sum_{K \in \bar f_0^{-1}(I)} \nu(K)\\
        &= w_0(I)
    \end{align*}
    which establishes Equation \eqref{eq:alt-matching1}.
    Equation \eqref{eq:alt-matching2} follows from the same argument.
    Finally, Equation \eqref{eq:alt-matching3} implies
    \begin{align*}
        ||\eta||' &= \max_{(I,J) \in \supp(\eta)} d_{\bar P} (I, J)\\
        &= \max_{K \in \supp(\nu)} d_{\bar P} \big(\bar f_0(K), \bar f_1(K) \big)\\
        &= ||\nu||.
    \end{align*}

    While both approaches to the bottleneck distance are equivalent, we focus primarily on the one in Definition \ref{def:bottleneck-distance} for two main reasons.
    First, it more closely parallels Galois interleavings; both Galois interleavings and matchings are spans in their respective categories.
    Second, allowing matchings to have arbitrary interval posets as domains provides extra flexibility that we exploit in our proof of the bottleneck stability theorem (Theorem~\ref{thm:bottleneck stability}).
    Ultimately, both approaches have their advantages and we exploit the equivalence between them when helpful.
\end{rmk}

\begin{rmk}\label{rmk: matching is insensitive to diagonal}
    The values of matchings along the diagonal can be altered arbitrarily while preserving the matching condition and maintaining a bound on the cost.
    Explicitly, let $w_0, w_1: \bar P \to \Z$ and $\nu: \bar R \to \Z$ be persistence diagrams with a matching 
    \[
    \begin{tikzcd}
            &\nu \ar[dl, "f_0"'] \ar[dr, "f_1"] \\
        w_0 &   &w_1.
    \end{tikzcd}
    \]
    If $\gamma: \bar R \to \Z$ is a nonnegative, finitely supported function with $\gamma(I) = \nu(I)$ for all $I \in \bar R^\circ$ then
    \[
    \begin{tikzcd}
            &\gamma \ar[dl, "f_0"'] \ar[dr, "f_1"] \\
        w_0 &   &w_1
    \end{tikzcd}
    \]
    is also a matching between $w_0$ and $w_1$ with $||\gamma|| \leq d_\infty(f_0, f_1)$.
    This follows immediately from the fact that $\bar f_0$ and $\bar f_1$ are defined componentwise and therefore map $\Delta_R$ into $\Delta_P$.
\end{rmk}

\begin{ex}\label{ex: neg-matchings}
    If one allowed matchings to take negative values, the bottleneck distance between any two persistence diagrams would always be $0$.
    For the sake of clarity, we use the equivalent notion of matching from Remark \ref{rmk:matchings-equivalence} here.
    Let $w_0: \overline{[0,\infty)} \to \Z$ be the persistence diagram with $w_0(0,5)=1$ and 0 elsewhere as shown in Figure \ref{fig: negative-matchings}.
    Similarly, let $w_1: \overline{[0,\infty)} \to \Z$ be the persistence diagram with $w_1(3,5)=1$ and 0 elsewhere.
    We construct a function $\eta: \overline{ [0,\infty)} \times \overline{ [0,\infty)} \to \Z$ by 
    \begin{align*}
        \eta \big( (0,5), (1,5) \big) &=1\\
        \eta \big( (1,5), (2,5) \big) &=-1\\
        \eta \big( (2,5), (3,5) \big) &=1
    \end{align*}
    and 0 elsewhere.
    This function $\eta$ satisfies all of the properties necessary to induce a matching between $w_0$ and $w_1$ apart from nonnegativity.
    The cost of $\eta$ is
    \[
        ||\eta||' = \max_{(I,J) \in \supp(\eta)} d_{\bar P}(I,J) = 1.
    \]
    Moreover, this process can be refined, breaking up the line segment between $(0,5)$ and $(3,5)$ into arbitrarily small intervals, producing ``matchings" with arbitrarily small costs. Hence, the bottleneck distance between $w_0$ and $w_1$ would be $0$.
    \begin{figure}
        \centering
        \includegraphics[scale=1]{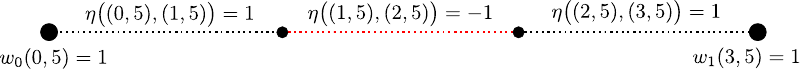}
        \caption{This example shows that nonnegativity is crucial in the definition of a matching. See Example \ref{ex: neg-matchings}.}
        \label{fig: negative-matchings}
    \end{figure}
\end{ex}

\subsection{The Triangle Inequalities}\label{subsec: triangle ineq}
In this section, we show that the Galois interleaving and bottleneck distances satisfy the triangle inequalities.

\begin{prop}[Triangle inequality for the Galois interleaving distance]
    Let $P_0, P_1, P_2$ be posets embedded into some metric space $(X,d_X)$ and let $\Mfunc_0: P_0 \to \Vec$, $\Mfunc_1:P_1 \to \Vec$, and $\Mfunc_2: P_2 \to \Vec$ be persistence modules.
    Then $d_I^G(\Mfunc_0, \Mfunc_2) \leq d_I^G(\Mfunc_0,\Mfunc_1) + d_I^G(\Mfunc_1, \Mfunc_2)$.
\end{prop}

\begin{proof}
    Suppose there are Galois interleavings
    \begin{center}
        \begin{tikzcd}
                &\Gamma_0\ar[dl]\ar[dr]   &   &\Gamma_1\ar[dl]\ar[dr]\\
            \Mfunc_0    &   &\Mfunc_1   &   &\Mfunc_2.
        \end{tikzcd}
    \end{center}
    Let $Q_0$ be the domain of $\Gamma_0$ and $Q_1$ be the domain of $\Gamma_1$.
    By the definition of Galois interleavings, we have Galois insertions
    \begin{center}
        \begin{tikzcd}
                &Q_0\ar[dr, "f_0"', bend right]\ar[dl, bend left, "f_{-1}"]     &   &Q_1\ar[dl, "f_1", bend left]\ar[dr, bend right, "f_2"']\\
             P_0\ar[ur, bend left, "g_{-1}"]    &   &P_1\ar[ul, "g_0"', bend right]\ar[ur, "g_1", bend left]    &   &P_2\ar[ul, bend right, "g_2"']
        \end{tikzcd}
    \end{center}
    such that 
    \begin{align*}
    \Gamma_0 \circ g_{-1} &\cong \Mfunc_0\\
    \Gamma_0 \circ g_0 &\cong \Mfunc_1 \cong \Gamma_1 \circ g_1\\
    \Gamma_1 \circ g_2 &\cong \Mfunc_2.
    \end{align*}
    Fix a natural isomorphism $\eta_{0,1}: \Gamma_0 \circ g_0 \to \Gamma_1 \circ g_1$ and let $\eta_{1,0}$ be its inverse.
    Recall from Definition~\ref{def: pushout-poset} the preordered set $Q_0 \sqcup_{P_1} Q_1$
    and from Proposition \ref{prop: concatenation insertion}, the Galois insertions $(\pi_0,\iota_0)$ and $(\pi_1, \iota_1)$
    \begin{center}
        \begin{tikzcd}
                &   &Q_0 \sqcup_{P_1} Q_1\ar[dl, bend left, "\pi_0"']\ar[dr, bend right, "\pi_1"]\\
                &Q_0\ar[dr, "f_0"', bend right]\ar[dl, bend left, "f_{-1}"]\ar[ur, bend left, "\iota_0"]     &   &Q_1\ar[dl, "f_1", bend left]\ar[dr, bend right, "f_2"']\ar[ul, bend right, "\iota_1"']\\
             P_0\ar[ur, bend left, "g_{-1}"]    &   &P_1\ar[ul, "g_0", bend right]\ar[ur, "g_1"', bend left]    &   &P_2.\ar[ul, bend right, "g_2"']
        \end{tikzcd}
    \end{center}
    Additionally, recall from Proposition \ref{prop: concatenation insertion} that the downwards arrows in the above diagram commute.
    That is, $f_0 \circ \pi_0 = f_1 \circ \pi_1$.
    
    We define a functor
    \[
        \Phi: Q_0 \sqcup_{P_1} Q_1 \to \Vec
    \]
    that induces a Galois interleaving between $\Mfunc_0$ and $\Mfunc_2$.
    For any $(a,t) \in Q_0 \sqcup_{P_1} Q_1$, let~${\Phi (a,t) = \Gamma_t(a)}$ and for any $(b,s) \succeq (a,t)$ we define the map from $\Phi(a,t)$ to $\Phi(b,s)$ by
    \begin{itemize}
        \item if $t=s$, then $\Phi \big( (a,t) \preceq (b,t) \big) = \Gamma_t ( a \leq b )$ or
        \item if $t \neq s$, then define $\Phi\big( (a,t) \preceq (b,s) \big)$ as the composition of the maps
    \end{itemize}
    \begin{center}
        \begin{tikzcd}
            \Gamma_t(a) \ar[r]  &\Gamma_t \Big( g_t \big( f_t (a) \big) \Big) \ar[r,"\eta_{t,s}"]    &\Gamma_s \Big( g_s\big( f_t(a) \big) \Big) \ar[r]  &\Gamma_s(b).
        \end{tikzcd}
    \end{center}
    To establish that $\Phi$ is a functor, we need to show that if $(a,t) \preceq (b,s) \preceq (c,u)$ then
    \begin{equation}\label{eq: functoriality}
        \Phi \big( (b,s) \preceq (c,u) \big) \circ \Phi \big( (a,t) \preceq (b,s) \big) = \Phi \big( (a,t) \preceq (c,u) \big).
    \end{equation}
    The case where $t=s=u$ follows immediately from the functoriality of $\Gamma_t$.
    If $t=s \neq u$, then Equation \eqref{eq: functoriality} follows from the commutative diagram
    \begin{center}
        \begin{tikzcd}
            \Gamma_t(b) \ar[r]  &\Gamma_t \Big( g_t \big( f_t (b) \big) \Big) \ar[r,"\eta_{t,u}"]    &\Gamma_u \Big( g_u\big( f_t(b) \big) \Big) \ar[r]  &\Gamma_u(c)\\
            \Gamma_t(a) \ar[r]\ar[u]  &\Gamma_t \Big( g_t \big( f_t (a) \big) \Big) \ar[r,"\eta_{t,u}"]\ar[u]    &\Gamma_u \Big( g_u\big( f_t(a) \big) \Big) \ar[r]\ar[u]  &\Gamma_u(c). \ar[u,"\cong"]
        \end{tikzcd}
    \end{center}
    Here the left and right squares commute by the functoriality of $\Gamma_t$ and $\Gamma_u$ respectively while the middle square commutes since $\eta_{t,u}$ is a natural transformation.
    The case where $t\neq s =u$ follows from a similar argument.

    If $t=u \neq s$ then the following commutative diagram implies Equation \eqref{eq: functoriality}
    \begin{center}
        \begin{tikzcd}
            &\Gamma_t \Big( g_t \big( f_t(a) \big) \Big)\ar[rr]   &    &\Gamma_t \Big( g_t \big( f_s(b) \big) \Big) \ar[r]  &\Gamma_t(c)\\
            &\Gamma_s \Big( g_s \big( f_t(a) \big) \Big) \ar[r]\ar[u,"\eta_{t,s}^{-1}"] &\Gamma_s (b) \ar[r]    &\Gamma_s \Big( g_s \big( f_s(b) \big) \Big) \ar[u,"\eta_{t,s}^{-1}"]\\
            \Gamma_t(a) \ar[r]   &\Gamma_t \Big( g_t \big( f_t(a) \big) \Big)\ar[rr]\ar[u,"\eta_{t,s}"]   &    &\Gamma_t \Big( g_t \big( f_s(b) \big) \Big).\ar[u,"\eta_{t,s}"]
        \end{tikzcd}
    \end{center}
    The top and bottom paths in this diagram both reduce to $\Gamma_t (a \leq_t c) = \Phi \big( (a,t) \preceq (c,t) \big)$ while the middle path is $\Phi \big( (b,s) \preceq (c,t) \big) \circ \Phi \big( (a,t) \preceq (b,s) \big)$.
    This proves that $\Phi$ is a functor.
    By construction, we have $\Phi \circ \iota_t \cong \Gamma_t$ for $t \in \{0,1\}$.
    Next we deal with the subtlety of $Q_0 \sqcup_{P_1} Q_1$ being a preordered set rather than a poset.

    Every preordered set (viewed as a category) is equivalent to a poset.
    Thus, there exists a poset $R$ and order-preserving maps $\Ffunc: Q_0 \sqcup_{P_1} Q_1 \to R$ and $\Gfunc: R \to Q_0 \sqcup_{P_1} Q_1$ with~${\Ffunc \circ \Gfunc \cong \id_R}$ and $\Gfunc \circ \Ffunc \cong \id_{Q_0 \sqcup_{P_1} Q_1}$.
    Define the persistence module $\Psi: R \to \Vec$ as $\Psi = \Phi \circ \Gfunc$.
    Since equivalences of categories are both left and right adjoints, the order preserving maps
    \begin{align*}
        f_{-1} \circ \pi_0 \circ \Gfunc:R &\leftrightarrows P_0: \Ffunc \circ \iota_0 \circ g_{-1}\\
        f_{2} \circ \pi_1 \circ \Gfunc:R &\leftrightarrows P_2: \Ffunc \circ \iota_1 \circ g_{2}
    \end{align*}
    form Galois insertions.
    That these Galois insertions form a Galois interleaving follows from
    \begin{align*}
        \Psi \circ \Ffunc \circ \iota_0 \circ g_{-1} &= \Phi \circ \Gfunc \circ \Ffunc \circ \iota_0 \circ g_{-1}\\
        &\cong \Phi \circ \iota_0 \circ g_{-1}\\
        &\cong \Gamma_0 \circ g_{-1}\\
        &\cong \Mfunc_0
    \end{align*}
    together with the analogous argument that $\Psi \circ \Ffunc \circ \iota_1 \circ g_2 \cong \Mfunc_2$.
    Therefore, $\Psi$ forms a Galois interleaving between $\Mfunc_0$ and $\Mfunc_2$.
    The cost of this interleaving is
    \begin{align*}
        ||\Psi|| &= d_\infty (f_{-1} \circ \pi_0 \circ \Gfunc, f_2 \circ \pi_1 \circ \Gfunc )\\
        &\leq d_\infty (f_{-1} \circ \pi_0, f_2 \circ \pi_1)\\
        &\leq d_\infty (f_{-1} \circ \pi_0, f_0 \circ \pi_0) + d_\infty (f_0 \circ \pi_0, f_2 \circ \pi_1)\\
        &= d_\infty (f_{-1} \circ \pi_0, f_0 \circ \pi_0) + d_\infty (f_1 \circ \pi_1, f_2 \circ \pi_1)\\
        &\leq d_\infty (f_{-1}, f_0) + d_\infty (f_1, f_2)\\
        &= ||\Gamma_0|| + ||\Gamma_1||.
    \end{align*}
\end{proof}

Now we show that the bottleneck distance also satisfies the triangle inequality.
See Example \ref{ex: composing-matchings} for an illustration.
This result is analogous to the gluing lemma in measure theory \cite{villani-optimal-transport}.

\begin{figure}
    \centering
    \includegraphics[scale=1]{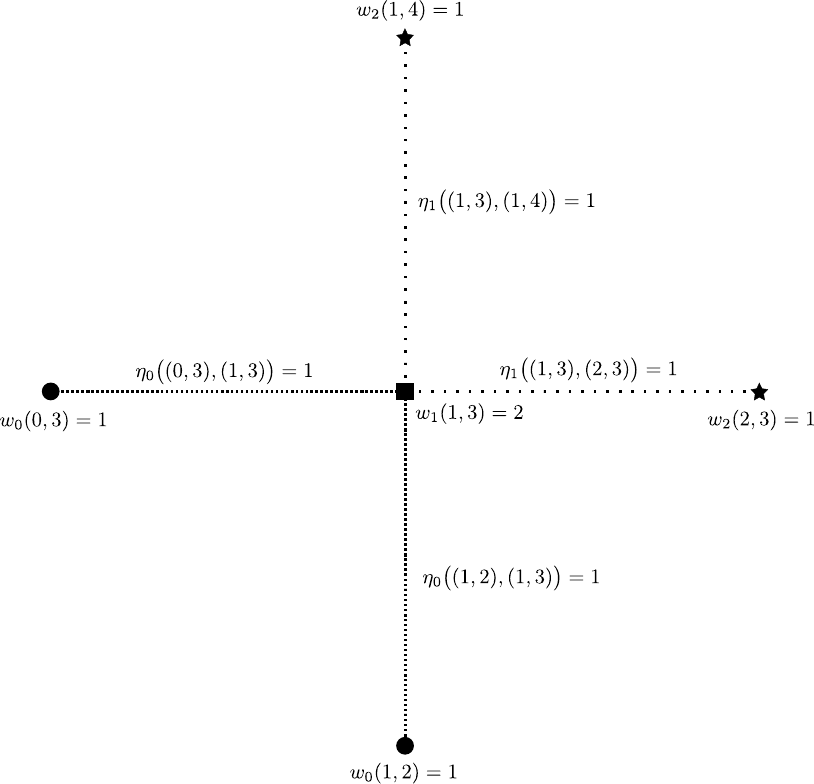}
    \caption{
        Above are three persistence diagrams $w_0$, $w_1$, and $w_2$ with functions $\eta_0,\eta_1$.
        By Remark~\ref{rmk:matchings-equivalence}, $\eta_0$ induces a matching between $w_0$ and $w_1$ and $\eta_1$ induces a matching between $w_1$ and $w_2$.
        There is no unique way to glue $\eta_0$ and $\eta_1$ together.
        Instead, our algorithm proceeds by choosing $(1,3)$ in the support of $w_1$ and choosing  points (say, $(0,3)$ and $(2,3)$) matched to $(1,3)$ by $\eta_0$ and $\eta_1$ respectively.
        The algorithm then matches $(0,3)$ with $(2,3)$ and removes them from the picture.
        This process continues until the original diagrams are empty.
    }
    \label{fig: composing-matchings}
\end{figure}

\begin{lem} \label{lem:pushforward-matchings}
    Let $(P, d_P)$ be a metric poset and let $w_0, w_1, w_2: \bar P \to \Z$ be persistence diagrams with matchings $w_0 \xleftarrow{f_0} \nu_0 \xrightarrow{f_1} w_1$ and $w_1 \xleftarrow{f_2} \nu_1 \xrightarrow{f_2} w_2$.
    Then there exist persistence diagrams $u_0, u_1, u_2: \bar P \to \Z$ and finitely supported, nonnegative functions $\eta_0, \eta_1: \bar P \times \bar P \to \Z$ such that,
    \begin{enumerate}
        \item  for any $i \in \{0,1,2\}$, $u_i|_{\bar P^\circ} = w_i|_{\bar P^\circ}$, \label{enum:matching1}
        \item for any $I \in \bar P$ (including $\Delta_P$),
        \begin{align*}
            u_0(I) &= \sum_{J \in \bar P} \eta_0(I, J)\\
            u_1(I) &= \sum_{J \in \bar P} \eta_0(J, I) = \sum_{J \in \bar P} \eta_1(I, J)\\
            u_2(I) &= \sum_{J \in \bar P} \eta_1(J, I),
        \end{align*}
        \label{enum:matching2}
        \item
        \[
            \max_{(I,J) \in \supp(\eta_0)} d_{\bar P}(I, J) = ||\nu_0|| \quad \text{and} \quad \max_{(I,J) \in \supp(\eta_1)} d_{\bar P}(I, J) = ||\nu_1||.
        \] \label{enum:matching3}
    \end{enumerate}
\end{lem}
\begin{proof}
    By Remark \ref{rmk:matchings-equivalence}, the matchings $\nu_0$ and $\nu_1$ induce finitely supported, nonnegative functions $\eta_0', \eta_1': \bar P \times \bar P \to \Z$ such that, for any $I \in \bar P^\circ$,
    \begin{align*}
            w_0(I) &= \sum_{J \in \bar P} \eta_0'(I, J)\\
            w_1(I) &= \sum_{J \in \bar P} \eta_0'(J, I) = \sum_{J \in \bar P} \eta_1'(I, J)\\
            w_2(I) &= \sum_{J \in \bar P} \eta_1'(J, I),
    \end{align*}
    and
    \begin{align*}
        \max_{(I,J) \in \supp(\eta_0')} d_{\bar P}(I, J) &= ||\nu_0||\\
        \max_{(I,J) \in \supp(\eta_1')} d_{\bar P}(I, J) &= ||\nu_1||.
    \end{align*}
    Define $u_1: \bar P \to \Z$ as
    \[
        u_1(I) = \max \Big\{ \sum_{J \in \bar P} \eta_0'(J, I),\, \sum_{J \in \bar P} \eta_1'(I, J) \Big\}.
    \]
    For any $I \in \bar P^\circ$, both terms in the maximum coincide with $w_1(I)$ so $u_1|_{\bar P^\circ} = w_1|_{\bar P^\circ}$.
    Since $u_1$ is the maximum of two nonnegative, finitely supported functions, it is also nonnegative and finitely supported.
    
    Next, we define the function $\kappa_0: \bar P \times \bar P \to \Z$ as
    \[
        \kappa_0(I,J) =
        \begin{cases}
            u_1(I) - \displaystyle{\sum_{J \in \bar P}} \eta_0'(J, I)  &\text{if } I=J,\\
            0   &\text{otherwise.}
        \end{cases}
    \]
    Observe that $\kappa_0$ is nonnegative and supported on a finite subset of $\Delta_P \times \Delta_P$.
    Now define~$\eta_0$ as ${\eta_0 = \eta_0' + \kappa_0}$.
    By construction, $\eta_0$ is finitely supported, nonnegative, and satisfies
    \[
        u_1(I) = \sum_{J \in \bar P} \eta_0(J,I).
    \]
    Because $\eta_0(I, J) = \eta_0'(I,J)$ whenever $I \neq J$, we have that
    \[
        \max_{(I,J) \in \supp(\eta_0)} d_{\bar P}(I,J) = \max_{(I,J) \in \supp(\eta_0')} d_{\bar P} (I, J) = ||\nu_0||.
    \]
    Next, let $u_0: \bar P \to \Z$ be defined as
    \[
        u_0(I) = \sum_{J \in \bar P} \eta_0(I, J).
    \]
    Then, because $\kappa_0$ is supported on a subset of $\Delta_P \times \Delta_P$, we have that for any $I \in \bar P^\circ$,
    \[
        u_0(I) = \sum_{J \in \bar P} \eta_0(I, J) =
        \sum_{J \in \bar P} \eta_0'(I, J) + \sum_{J \in \bar P} \kappa_0(I, J) =
        \sum_{J \in \bar P} \eta_0'(I, J) = 
        w_0(I).
    \]
    Finally, constructing $\eta_1$ and $u_2$ with the same method concludes the proof.
\end{proof}

\begin{prop}[Triangle inequality for the bottleneck distance]\label{prop: triangle for bottleneck}
    Let $(P,d_P)$ be a metric poset and let $w_0, w_1, w_2: \bar P \to \Z$ be persistence diagrams with matchings $w_0 \xleftarrow{f_0} \nu_0 \xrightarrow{f_1} w_1$ and $w_1 \xleftarrow{f_2} \nu_1 \xrightarrow{f_3} w_2$.
    Then, there is a matching $\gamma$ between $w_0$ and $w_2$ with cost $||\gamma|| \leq ||\nu_0|| + ||\nu_1||$.
\end{prop}
\begin{proof}
    By Lemma \ref{lem:pushforward-matchings}, there exist persistence diagrams $u_0, u_1, u_2: \bar P \to \Z$ and finitely supported, nonnegative functions $\eta_0, \eta_1: \bar P \times \bar P \to \Z$ satisfying \ref{enum:matching1}, \ref{enum:matching2}, and \ref{enum:matching3}.
    Note that \ref{enum:matching2} implies that $u_0,u_1$, and $u_2$ are nonnegative and that summing each of the functions $u_0, u_1, u_2, \eta_0$, and $\eta_1$ over their domains yields the same number.
    Let $N$ be this number.

    For every $i \in \{0, 1, \ldots, N\}$, we inductively construct functions $\eta^{(i)}, \eta_0^{(i)}, \eta_1^{(i)}: \bar P \times \bar P \to \Z$ and $u_0^{(i)}, u_1^{(i)}, u_2^{(i)}: \bar P \to \Z$ such that,
    \begin{enumerate}
        \item for any $I \in \bar P$,
        \begin{align*}
            u_0^{(i)}(I) &= \sum_{J \in \bar P} \eta_0^{(i)}(I, J) = u_0(I) - \sum_{J \in \bar P} \eta^{(i)}(I, J)\\
            u_1^{(i)}(I) &= \sum_{J \in \bar P} \eta_0^{(i)}(J, I) = \sum_{J \in \bar P} \eta_1^{(i)}(I, J)\\
            u_2^{(i)}(I) &= \sum_{J \in \bar P} \eta_1^{(i)}(J, I) = u_2(I) - \sum_{J \in \bar P} \eta^{(i)}(J, I),
        \end{align*} \label{enum:bottleneck1}
        \item $\supp (\eta_0^{(i)}) \subseteq \supp(\eta_0)$, $\supp (\eta_1^{(i)}) \subseteq \supp(\eta_1)$, and \label{enum:bottleneck2}
        \item $\max_{(I,J) \in \supp(\eta^{(i)})} d_{\bar P}(I, J) \leq ||\nu_0|| + ||\nu_1||$.\label{enum:bottleneck3}
    \end{enumerate}

    Start with the functions  defined by
    \begin{align*}
        &\eta^{(0)} := 0
        &u_0^{(0)} := u_0\\
        &\eta_0^{(0)} := \eta_0
        &u_1^{(0)} := u_1\\
        &\eta_1^{(0)} := \eta_1
        &u_2^{(0)} := u_2.
    \end{align*}
    These trivially satisfy \ref{enum:bottleneck1}, \ref{enum:bottleneck2}, and \ref{enum:bottleneck3}.
    Now choose $i\in \{ 0, 1,...,N-1\}$ and suppose that \ref{enum:bottleneck1}, \ref{enum:bottleneck2}, and~\ref{enum:bottleneck3} are satisfied.
    Pick some $I_1 \in \supp(u_1^{(i)})$ and then choose $I_0, I_2 \in \bar P$ with $\eta_0^{(i)}(I_0, I_1) > 0$ and $\eta_1^{(i)}(I_1, I_2) > 0$.
    Let
    \begin{align*}
        &\eta^{(i+1)} := \eta^{(i)} + \mathbf{1}_{(I_0, I_2)}
        &u_0^{(i+1)} := u_0^{(i)} - \mathbf{1}_{I_0}\\
        &\eta_0^{(i+1)} := \eta_0^{(i)} - \mathbf{1}_{(I_0, I_1)}
        &u_1^{(i+1)} := u_1^{(i)} - \mathbf{1}_{I_1}\\
        &\eta_1^{(i+1)} := \eta_1^{(i)} - \mathbf{1}_{(I_1, I_2)}
        &u_2^{(i+1)} := u_2^{(i)} - \mathbf{1}_{I_2}
    \end{align*}
    where $\mathbf{1}_{x}$ denotes the indicator function on $\{x\}$.
    To establish that these functions satisfy \ref{enum:bottleneck1}, we note that all of the equations follow from similar arguments so, for the sake of brevity, we only show the first two equalities here.
    For any $I \in \bar P$,
    \begin{align*}
        u_0^{(i+1)}(I) &= u_0^{(i)}(I) - \mathbf{1}_{I_0}(I)\\
        &= \sum_{J \in \bar P} \eta_0^{(i)}(I,J) - \mathbf{1}_{I_0}(I)\\
        &= \sum_{J \in \bar P} \eta_0^{(i+1)}(I,J) + \sum_{J \in \bar P} \mathbf{1}_{(I_0, I_1)}(I,J) - \mathbf{1}_{I_0}(I)\\
        &= \sum_{J \in \bar P} \eta_0^{(i+1)}(I,J) + \mathbf{1}_{I_0}(I) - \mathbf{1}_{I_0}(I)\\
        &= \sum_{J \in \bar P} \eta_0^{(i+1)}(I,J)
    \end{align*}
    and
    \begin{align*}
        u_0^{(i+1)}(I) &= u_0^{(i)}(I) - \mathbf{1}_{I_0}(I)\\
        &= u_0(I) - \sum_{J \in \bar P} \eta^{(i)}(I,J) - \mathbf{1}_{I_0}(I)\\
        &= u_0(I) - \Big( \sum_{J \in \bar P} \eta^{(i+1)}(I,J) - \sum_{J \in \bar P} \mathbf{1}_{(I_0, I_2)} (I,J) \Big) - \mathbf{1}_{I_0}(I)\\
        &= u_0(I) - \sum_{J \in \bar P} \eta^{(i+1)}(I,J) + \mathbf{1}_{I_0}(I) - \mathbf{1}_{I_0}(I)\\
        &= u_0(I) - \sum_{J \in \bar P} \eta^{(i+1)}(I,J).
    \end{align*}
    Property \ref{enum:bottleneck2} is immediate from construction.
    To establish \ref{enum:bottleneck3}, first note that
    \[
        \supp(\eta^{(i+1)}) = \supp(\eta^{(i)}) \cup \{(I_0, I_2)\}
    \]
    so it suffices to show that $d_{\bar P}(I_0, I_2) \leq ||\nu_0|| + ||\nu_1||$.
    Because $(I_0, I_1) \in \supp(\eta_0^{(i)}) \subseteq \supp(\eta_0)$, we have
    \[
        d_{\bar P}(I_0, I_1) \leq \max_{(I,J) \in \supp(\eta_0)} d_{\bar P} (I, J) = ||\nu_0||
    \]
    and likewise, $d_{\bar P} (I_1, I_2) \leq ||\nu_1||$.
    Therefore, by the triangle inequality for $d_{\bar P}$, we have that $d_{\bar P}(I_0 , I_2) \leq ||\nu_0|| + ||\nu_1||$.

    By construction, we have that $u_0^{(N)}(I) = u_1^{(N)}(I) = u_2^{(N)}(I) = 0$ for any $I \in \bar P$.
    Combining this with \ref{enum:bottleneck1} shows that, for any $I \in \bar P$,
    \begin{align*}
        u_0(I) &= \sum_{J \in \bar P} \eta^{(N)}(I,J)\\
        u_2(I) &= \sum_{J \in \bar P} \eta^{(N)}(J,I).
    \end{align*}
    By \ref{enum:bottleneck3}, we have that
    \[
        \max_{(I,J) \in \supp(\eta^{(N)})} d_{\bar P} (I,J) \leq ||\nu_0|| + ||\nu_1||.
    \]
    Therefore, by Remark \ref{rmk:matchings-equivalence}, $\eta^{(N)}$ induces a matching $u_0\leftarrow\gamma \rightarrow u_2$ with $||\gamma|| \leq ||\nu_0|| + ||\nu_1||$.
    Because $u_0$ and $u_2$ coincide with $w_0$ and $w_2$ respectively on $\bar P^\circ$, $\gamma$ is a matching between $w_0$ and $w_2$ as well (see Remark \ref{rmk: matching is insensitive to diagonal}).
\end{proof}

\begin{ex}\label{ex: composing-matchings}
    Let $w_0, w_1, w_2: \overline{[0,\infty)} \to \Z$ be the persistence diagrams defined as
    \begin{align*}
        w_0(0,3) &= w_0(1,2)=1\\
        w_1(1,3) &= 2\\
        w_2(1,4) &= w_2(2,3)=1
    \end{align*}
    and 0 elsewhere.
    Define the functions $\eta_0, \eta_1: \overline{[0,\infty)} \times \overline{[0,\infty)} \to \Z$ as
    \begin{align*}
        \eta_0 \big( (0,3),(1,3) \big) &= \eta_0 \big( (1,2),(1,3) \big)=1\\
        \eta_1 \big( (1,3), (1,4) \big) &= \eta_1 \big( (1,3), (2,3) \big)=1
    \end{align*}
    and taking value 0 outside of these pairs.
    By Remark \ref{rmk:matchings-equivalence}, $\eta_0$ induces a matching between~$w_0$ and $w_1$, and $\eta_1$ induces a matching between $w_1$ and $w_2$.
    The proof of Proposition \ref{prop: triangle for bottleneck} constructs a function $\eta^{(2)}: \overline{[0,\infty)} \times \overline{[0,\infty)} \to \Z$ which induces a matching between $w_0$ and~$w_2$.
    See Figure~\ref{fig: composing-matchings} for an illustration.
\end{ex}

\section{Bottleneck Stability}\label{sec: stability}

The bottleneck stability theorem, originally proved in \cite{CSEdH}, is arguably the most important result in persistent homology.
Here we present a new and simpler proof using Rota's Galois connection theorem.

\begin{thm}[Bottleneck Stability]\label{thm:bottleneck stability}
    Let $\Mfunc_0 , \Mfunc_1: [0,\infty) \to \Vec$ be constructible persistence modules with persistence diagrams $m_0 := \partial \mathsf{null}_{\Mfunc_0}$ and $m_1 := \partial \mathsf{null}_{\Mfunc_1}$ respectively. Then, $$d_B (m_0, m_1) \leq d_I(\Mfunc_0, \Mfunc_1).$$
\end{thm}

It may appear that functoriality (Proposition \ref{prop:functoriality}) immediately implies the result since a Galois interleaving $\Mfunc_0 \leftarrow \Gamma \rightarrow \Mfunc_1$ induces morphisms $\partial \mathsf{null}_{\Mfunc_0} \leftarrow \partial \mathsf{null}_{\Gamma} \rightarrow \partial \mathsf{null}_{\Mfunc_1}$.
Unfortunately though, we cannot guarantee that $\partial \mathsf{null}_{\Gamma}$ is nonnegative.
Instead, we proceed by breaking up the Galois interleaving $\Gamma$ into simpler Galois interleavings whose persistence diagrams we can guarantee are nonnegative.
We accomplish this by using a Galois interleaving to interpolate between $\Mfunc_0$ and $\Mfunc_1$ (Lemma \ref{lem:interpolation finite}).
We then use this interpolation to obtain a sequence of persistence modules interleaved by persistence modules over totally ordered sets.
Proposition \ref{prop: positivity} then guarantees that each of these interleavings induces a matching.

Before proving Theorem~\ref{thm:bottleneck stability}, we first show that any Galois interleaving between $\Mfunc_0$ and~$\Mfunc_1$ with cost $\varepsilon$ can be captured by a Galois interleaving between persistence modules over finite posets with cost $\varepsilon$.

\begin{lem}\label{lem: finite poset repn}
    Let $\Mfunc_0, \Mfunc_1: [0,\infty) \to \Vec$ be constructible persistence modules and let $\Gamma$ be a Galois interleaving between $\Mfunc_0$ and $\Mfunc_1$ with cost $\ee$.
    Then, there exist finite subsets $S_0,S_1 \subseteq [0,\infty)$ and persistence modules $\Tilde{\Mfunc}_0 : S_0 \to \Vec$ and $\tilde{\Mfunc}_1: S_1 \to \Vec$ together with a finite poset $R$ and a Galois interleaving $\Tilde{\Gamma} : R \to \Vec$ between $\Tilde{\Mfunc}_0$ and $\Tilde{\Mfunc}_1$ with cost at most $\varepsilon$.
    Moreover, for $i \in \{0,1\}$, if $m_i$ and $\tilde{m_i}$ are persistence diagrams of $\Mfunc_i$ and $\tilde{\Mfunc}_i$ respectively, then, for any $I \in \overline{[0,\infty)}^\circ$,
    \begin{align*}
        m_i(I) &=
        \begin{cases}
            \tilde{m_i}(I)  &\text{if } I \in \overline{S_i}\\
            0               &\text{otherwise}.
        \end{cases}
    \end{align*}

\end{lem}

\begin{proof}
   Let $T_0, T_1 \subseteq [0,\infty)$ be finite sets such that $\Mfunc_0$ is $T_0$-constructible and $\Mfunc_1$ is $T_1$-constructible. By Proposition~\ref{prop: old interleaving}, there is an $\varepsilon$-interleaving $\Gamma$ between $\Mfunc_0$ and $\Mfunc_1$ such that the following diagram commutes.

   \begin{center}
       \begin{tikzcd}
               &{[0,\infty)} \times_\ee \{0,1\}\ar[dd,"\Gamma"]\\
           {[0,\infty)}\ar[ur,"\iota_0",hook]\ar[dr,"\Mfunc_0"']  &   &{[0,\infty)}\ar[ul,"\iota_1"',hook]\ar[dl,"\Mfunc_1"]\\
               &\Vec
       \end{tikzcd}
   \end{center}

   The finiteness of $T_0$ and $T_1$, together with the $\varepsilon$-interleaving $\Gamma$, implies that there is $t \in [0,\infty)$ such that every $t' \geq t$, we have $\Mfunc_0 (t') \cong \Mfunc_1(t')$. We then define 
   \[
       S_i := \Big ( \big (\cup_{n=0}^{\infty} \big ( (T_i \cup (T_{1-i} + \varepsilon)) + n2\varepsilon \big ) \big ) \cap [0,t] \Big ) \cup \{t\}
   \]
   and $\Tilde{\Mfunc}_i := {\Mfunc_i}|_{S_i}$
   for $i=0,1$.
   Since each $\Mfunc_i$ is $T_i$ constructible and $T_i \subseteq S_i$, the persistence diagrams of $\Mfunc_i$ and $\Tilde{\Mfunc}_i$ have the same values on $\overline{S_i}^\circ$ and are 0 outside $\overline{S_i}$ by Proposition \ref{prop: constructible MI}.

   Observe that $ \iota_0 (S_0) \cup \iota_1 (S_1)$ has two maximal elements, namely $(t,0)$ and $(t,1)$. Gluing these two maximal elements via an equivalence relation $\sim$, let $R := \iota_0 (S_0) \cup \iota_1 (S_1) / \sim$ and  $\Tilde{\Gamma} = \Gamma|_{R}$. Observe that $\iota_i$ has a left adjoint $\pi_i$ given by $\pi_i(x,s) = \min (x+\ee |i-s|, t)$. Then, we have the following Galois interleaving 

   \begin{center}
       \begin{tikzcd}
               &R\ar[dl,"\pi_0"', bend right, dashed]\ar[dr, "\pi_1", bend left, dashed]\ar[dd,"\Tilde{\Gamma}"]\\
           {S_0}\ar[dr,"\Tilde{\Mfunc}_0"']\ar[ur, bend right,"\iota_0"]   &   &{S_1}\ar[dl,"\Tilde{\Mfunc}_1"]\ar[ul, bend left, "\iota_1"']\\
               &\Vec
       \end{tikzcd}
   \end{center}
    together with cost $||\Tilde{\Gamma}|| = d_\infty(\pi_0 , \pi_1) \leq \ee$, which concludes the proof.
\end{proof}

The following lemma, known as the interpolation lemma, appears in some form in many different proofs of the bottleneck stability theorem\cite{vineyards, CSEdH, skraba-turner-2020, bjerkevik-lesnick2021}.
We extend the interpolation lemma from the traditional setting of $\ee$-interleavings to the more general setting of Galois interleavings.

\begin{lem}[Interpolation]\label{lem:interpolation finite}
    Let $S_0, S_1 \subseteq [0,\infty)$ be finite posets and let $\Mfunc_0: S_0 \to \Vec$ and $\Mfunc_1: S_1 \to \Vec$ be two persistence modules with a Galois interleaving $\Gamma: R \to \Vec$, where $R$ is a finite poset.
    Suppose $||\Gamma|| = \ee$.
    Then there exists a family of persistence modules, $\{\Mfunc_t: S_t \to \Vec \}_{t\in [0,1]}$, such that $d_I^G (\Mfunc_t, \Mfunc_s) \leq \ee|t-s|$ for any $t,s\in [0,1]$.
\end{lem}
\begin{proof}
    Let
    \begin{center}
        \begin{tikzcd}
                &R\ar[dl,"f_0"', bend right, dashed]\ar[dr, "f_1", bend left, dashed]\ar[dd,"\Gamma"]\\
            {S_0}\ar[dr,"\Mfunc_0"']\ar[ur, bend right,"g_0"]   &   &{S_1}\ar[dl,"\Mfunc_1"]\ar[ul, bend left, "g_1"']\\
                &\Vec
        \end{tikzcd}
    \end{center}
    be the Galois interleaving between $\Mfunc_0$ and $\Mfunc_1$.
    Recall from Definition \ref{def:downsets} that $\down(R)$ is the poset of non-empty downsets of $R$.
    For each $t\in (0,1)$, define the order-preserving map $f_t : \down(R) \to [0,\infty)$ by 
    \[
         f_t(A) = \max_{a\in A} \{ (1-t) \, f_0(a) + t \, f_1(a) \}.
    \]
    Let $S_t := \image(f_t)$, and consider $f_t$ as a map into $S_t$. Then, as in Remark~\ref{rmk: unique left-right adjoint via formula}, each $f_t$ has a right adjoint $g_t : S_t \to \down(R)$ given by 
    \[
        g_t(x) =\{r \in R \mid (1-t) f_0(r) + t f_1(r) \leq x \}.
    \]
    We now extend $\Gamma: R \to \Vec$ to a persistence module $\tilde{\Gamma}: \down(R) \to \Vec$ by $\tilde{\Gamma} (A) := \colim \, \Gamma |_A$.
    Since $\Gamma(r) \cong \colim \Gamma |_{\downarrow r} $ for every $r\in R$, the following diagram commutes up to a natural isomorphism.
    \begin{center}
        \begin{tikzcd}
            R \ar[r, hook] \ar[rd, "\Gamma"']    & \down(R) \ar[d, "\tilde{\Gamma}"] \\
                & \Vec.
        \end{tikzcd}
    \end{center}
    For each $t\in (0,1)$ we define $\Mfunc_t := \tilde{\Gamma} \circ g_t$. That $d_I^G(\Mfunc_t, \Mfunc_s) \leq \varepsilon |t-s|$ follows from the commutative diagram
    \begin{center}
        \begin{tikzcd}
            & \down(R) \ar[ld, "f_t"', bend right, dashed] \ar[rd, "f_s", dashed, bend left] \ar[dd, "\tilde{\Gamma}"] \\
            {S_t} \arrow[ru, "g_t", bend right] \arrow[rd, "\Mfunc_t"']    &   & {S_s}\ar[lu, "g_s"', bend left] \ar[ld, "\Mfunc_s"]\\
            &\Vec
        \end{tikzcd}
    \end{center}
    together with the fact that $d_\infty (f_t, f_s) \leq \ee |t-s|$, which follows from:

    \begin{align*}
       d_\infty (f_t, f_s) &= \max_{A \in \down(R)} \big|f_t (A) - f_s(A)\big| \\
        &= \max_{A \in \down(R)} \big|\max_{a\in A}\{ (1-t)f_0(a)+t f_1(a) \} - \max_{a\in A}\{ (1-s)f_0(a)+sf_1(a) \}\big| \\
        &\leq \max_{A\in \down(R)} \max_{a\in A} \big| (1-t)f_0(a) + tf_1(a) - (1-s)f_0(a) - sf_1(a)   \big| \\ 
        &=\max_{r\in R} \big| (1-t)f_0(r) + tf_1(r) - (1-s)f_0(r) - sf_1(r)   \big| \\
        &= \max_{r\in R} \big| (s-t)f_0(r) - (s-t)f_1(r)   \big| \\
        &= |t-s| \cdot \max_{r\in R} \big| f_0(r) - f_1(r) \big| \\
        &= |t-s| \cdot ||\Gamma|| = \ee |t-s|.
    \end{align*}

\end{proof}

\begin{ex}\label{ex: interpolation}
    Let $\Mfunc_0 \leftarrow \Gamma \rightarrow \Mfunc_1$ be the Galois interleaving from Example \ref{ex: interleaving}.
    The interpolation induced by this Galois interleaving is shown in Figure \ref{fig: interpolation}.
    \begin{figure}
        \centering
        \includegraphics[scale=.8]{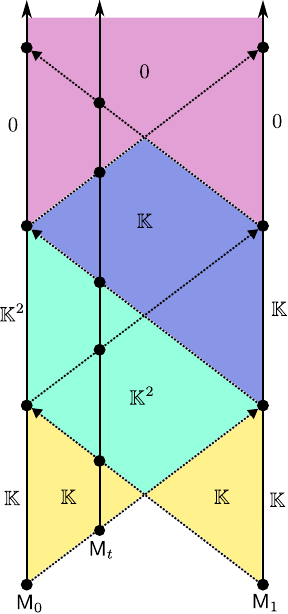}
        \caption{The interpolation of $\Mfunc_0$ and $\Mfunc_1$ induced by the Galois interleaving $\Gamma$ from Example \ref{ex: interleaving}. The domain $S_t$ of $\Mfunc_t$ is given by the intersection of the dashed lines with the vertical line at $t$. The critical points of the interpolation occur where the dashed lines intersect: at 0, .5, and 1. The maps $\alpha_{t,c}$ defined in Proposition \ref{prop:morphism} flow along the dashed lines.}
        \label{fig: interpolation}
    \end{figure}
\end{ex}

Within an interpolation $\{\Mfunc_t\}$, there are certain points where the fundamental structure of the persistence modules change.
We call these points critical.

\begin{defn}
    Let $\{ \Mfunc_t \}$ be an interpolation between persistence $\Mfunc_0$ and $\Mfunc_1$ with $R,f_0,f_1$ as defined in the proof of Lemma \ref{lem:interpolation finite}.
    A point $t \in [0,1]$ is \define{critical} if there exist distinct pairs $(f_0 (r), f_1 (r)) \neq (f_0 (r'),f_1(r'))$ for $r,r' \in R$ with
    \[(1-t) f_0(r) + t f_1(r) = (1-t) f_0 (r') + t f_1 (r').\]
\end{defn}

Because there are only finitely many distinct pairs $(f_0(r), f_1(r))$ for $r\in R$ and the map $t \mapsto (1-t)f_0(r) + t f_1(r)$ is linear for every $r\in R$, there are finitely many critical points in an interpolation.

\begin{prop}\label{prop:morphism}
    Let $\Mfunc_0: S_0 \to \Vec$ and $\Mfunc_1: S_1 \to \Vec$ be persistence modules over finite subsets $S_0, S_1 \subseteq [0, \infty)$, let $\Gamma: R \to \Vec$ be a Galois interleaving between them over a finite poset $R$, and let $\{\Mfunc_t\}$ be the interpolation induced by $\Gamma$.
    If $t\in [0,1]$ is a non-critical point and $c \in [0,1]$ is any (possibly critical) point with no critical points strictly between $t$ and $c$ then there is a morphism $\alpha_{t,c}: S_t \leftrightarrows S_c : \beta_{c,t}$ from $\Mfunc_t$ to $\Mfunc_c$.
\end{prop}
\begin{proof}
    We reuse the notation from the proof of Lemma \ref{lem:interpolation finite} here and begin by defining a poset map $\alpha_{t,c}: S_t \to S_c$.
    As $t$ is non-critical, each $x \in S_t$ can be written uniquely as $(1-t)f_0(r) + t f_1(r)$ for some $r\in R$.
    We can then define $\alpha_{t,c}$ as
    \[
        \alpha_{t,c}: x = (1-t)f_0(r) + t f_1(r) \mapsto (1-c) f_0(r) + c f_1(r).
    \]
    See Figure \ref{fig: interpolation} for an illustration.
    Since there are no critical points strictly between $t$ and~$c$, the map $\alpha_{t,c}$ is order-preserving and sends the minimum element of $S_t$ to the minimum element of $S_c$.
    Thus, by Lemma \ref{lem:finite-totally-ordered-adjoints}, $\alpha_{t,c}$ has a right adjoint $\beta_{c,t}: S_c \to S_t$.
    It follows that
    \begin{align*}
        g_t \big( \beta_{c,t}(x) \big)
        &= \{ r \in R \mid (1-t) f_0(r) + t f_1(r) \leq \beta_{c,t}(x) \}\\
        &= \{ r \in R \mid \alpha_{t,c} \big( (1-t) f_0(r) + t f_1(r) \big) \leq x \}\\
        &= \{ r \in R \mid (1-c) f_0(r) + c f_1(r) \leq x \}\\
        &= g_c(x).
    \end{align*}
    Therefore, $\tilde{\Gamma} \circ g_t \circ \beta_{c,t} \cong \tilde{\Gamma} \circ g_c$ and so $\Mfunc_t \circ \beta_{c,t} \cong \Mfunc_{c}$ which concludes the proof.
\end{proof}

We are now ready to prove the bottleneck stability theorem.

\begin{proof}[Proof of Theorem~\ref{thm:bottleneck stability}]
By Lemma~\ref{lem: finite poset repn}, we can assume that $\Mfunc_0 : S_0 \to \Vec$ and $\Mfunc_1 : S_1 \to \Vec$ are persistence modules over finite subsets $S_0, S_1 \subseteq [0,\infty)$ with a Galois interleaving $\Gamma : R \to \Vec$ between them with cost $\ee$ over some finite poset $R$.
By Lemma~\ref{lem:interpolation finite}, $\Gamma$ induces an interpolation $\{ M_t \}$ with critical points $C :=\{ 0=c_0 < c_1 < \cdots < c_k = 1 \} \subseteq [0,1]$.
Choose a point $t_i$ in the open interval $(c_i, c_{i+1})$ for each $i \in \{ 0,1,\ldots k-1 \}$.
By Proposition \ref{prop:morphism}, there are morphisms
\begin{center}
    \begin{tikzcd}
                &\Mfunc_{t_0}\ar[dl]\ar[dr]    &   &\Mfunc_{t_1}\ar[dl]\ar[dr]    &   &\ldots\ar[dl]\ar[dr]\\
        \Mfunc_0     &           &\Mfunc_{c_1}    &       &\Mfunc_{c_2}    &   &\Mfunc_1.
    \end{tikzcd}
\end{center}
Here the Galois interleaving $\Gamma$ breaks down into a sequence of Galois interleavings $\Mfunc_{t_i}$, each indexed by totally ordered sets.
By functoriality (Proposition \ref{prop:functoriality}), this induces a sequence of morphisms

\begin{center}
    \begin{tikzcd}[column sep = small]
                &{\partial \mathsf{null}_{\Mfunc_{t_0}}}\ar[dl]\ar[dr]    &   &{\partial \mathsf{null}_{\Mfunc_{t_1}}}\ar[dl]\ar[dr]    &   &\ldots\ar[dl]\ar[dr]\\
        {\partial \mathsf{null}_{\Mfunc_0}}     &           &{\partial \mathsf{null}_{\Mfunc_{c_1}}}    &       &{\partial \mathsf{null}_{\Mfunc_{c_2}}}    &   &{\partial \mathsf{null}_{\Mfunc_1}}.
    \end{tikzcd}
\end{center}
By Proposition \ref{prop: positivity}, each $\partial \mathsf{null}_{\Mfunc_{t_i}}$ is nonnegative on the off-diagonal intervals and, by Remark \ref{rmk: matching is insensitive to diagonal}, we can set all the diagonal values to zero while preserving the morphisms.
Moreover, we may extend the domains of the persistence diagrams $\partial \mathsf{null}_{\Mfunc_{c_i}}$ from $\overline{S_{c_i}}\subseteq \overline{[0,\infty)}$ to $\overline{[0,\infty)}$ by assigning zero values to any interval $I \in \overline{[0,\infty)}\setminus \overline{S_{c_i}}$.
Therefore, each span in the above diagram is a matching with cost
\begin{align*}
    ||\partial \mathsf{null}_{\Mfunc_{t_i}}|| &\leq d_\infty (\alpha_{t_i , c_i}, \alpha_{t_i , c_{i+1}})\\ 
    &= \max_{s \in S_{t_i}} \big| \alpha_{t_i , c_i}(s) - \alpha_{t_i , c_{i+1}}(s)  \big|\\
    &= \max_{r \in R} \big| (1-c_i) f_0 (r) +c_i f_1(r) - \big( (1-c_{i+1}) f_0(r) + c_{i+1} f_1(r) \big) \big|\\
        &= \max_{r \in R} \big| (c_{i+1} -c_i) f_0(r) - (c_{i+1} -c_{i} ) f_1(r) \big|\\
        &= (c_{i+1} - c_i) \max_{r \in R} |f_0(r)-f_1(r)|\\
        &\leq  (c_{i+1} - c_i) \ee.
\end{align*}
Now by the triangle inequality (Proposition~\ref{prop: triangle for bottleneck}), there is a matching between $\partial \mathsf{null}_{\Mfunc_0}$ and $\partial \mathsf{null}_{\Mfunc_1}$ with cost at most 
\[
    \sum_{i=0}^{k-1} (c_{i+1}-c_i)\ee = \ee.
\]
This concludes the proof of Theorem \ref{thm:bottleneck stability}.
\end{proof}

\section{Multiparameter Persistence Diagrams}\label{sec: multiparameter}

In this section, we extend the notion of birth-death functions, originally defined for filtrations~\cite{editdistance}, to persistence modules. We then demonstrate that the M\"obius inversion of birth-death functions coincide with the persistence diagram of the module, i.e. the M\"obius inversion of nullity functions; see Proposition~\ref{prop:ker and bd}. Lastly, we show how another application of Rota's Galois connection theorem facilitates the construction of fibered barcodes from persistence diagrams induced by nullity/birth-death functions.

\subsection{Birth-Death Functions}

In this section, we define birth-death functions of persistence modules using free covers.
We show that birth-death functions, like nullity functions, induce persistence diagrams.

\subsubsection{Free Persistence Modules}\label{sec: free-modules}

Free persistence modules are necessary to extend the birth-death function from filtrations~\cite{editdistance} to persistence modules.
Here we define free persistence modules and show that every constructible persistence module admits a (finite-dimensional) free cover.

Recall that for any poset $P$ and for any $a\in P$, the persistence module $\kk^{\uparrow a} : P \to \Vec$ is defined by
    \[
        \kk^{\uparrow a} (b):= 
        \begin{cases} 
            \kk &\text{ if }  a \leq b\\
            0   &\text{otherwise}
        \end{cases}
    \]
and $\kk^{\uparrow a}(b\leq c) = \id_\kk$ for all $a\leq b \leq c$.

\begin{defn}
    A persistence module $\Ffunc: P \to \Vec$ is \define{free} if it is isomorphic to a finite direct sum of persistence modules of the form $\kk^{\uparrow a}$ for some $a \in P$.
    A \define{free cover} of a persistence module $\Mfunc: P \to \Vec$ consists of a free persistence module $\Ffunc : P \to \Vec$ and a surjective natural transformation $\varphi: \Ffunc \Rightarrow \Mfunc$.
\end{defn}

Since we are working in the category $\Vec$ of finite dimensional vector spaces, the free covers here are always finite direct sums of persistence modules of the form $\kk^{\uparrow a}$.
This differs from the setting of \cite{miller} where infinite direct sums are allowed.

\begin{ex}\label{ex:free presentation}
    Figure \ref{fig:free} shows a free cover of the persistence module $\Mfunc: P \to \Vec$ appearing in Example~\ref{ex:module}.
    The natural transformation $\varphi: \Ffunc \Rightarrow \Mfunc$ is given by 
    \begin{align*}
        \varphi_a &: (x, y, z) \mapsto (x, y)\\
        \varphi_b &: w \mapsto w\\
        \varphi_c &: (x,y,z,w) \mapsto x + y + w\\
        \varphi_d &: (x,y,z,w) \mapsto 0.
    \end{align*}
\end{ex}

The following two lemmas show that free covers behave well with morphisms of persistence modules.

\begin{lem}\label{lem:free pullback 0}
    Let $f: P \leftrightarrows Q:g$ be a Galois connection between any posets $P$ and $Q$. Then $\kk^{\uparrow a} \circ g \cong \kk^{\uparrow f(a)}$ for any $a \in P$.
\end{lem}
\begin{proof}
    For any $q\in Q$, 
    \begin{align*}
        \big(\kk^{\uparrow a} \circ g \big) (q) &\cong
        \begin{cases}
            \kk   &\text{if } a \leq g(q)\\
            0   &\text{otherwise}
        \end{cases}\\
        &\cong
        \begin{cases}
            \kk   &\text{if } f(a) \leq q\\
            0   &\text{otherwise}
        \end{cases}\\
        &\cong \kk^{\uparrow f(a)}(q).
    \end{align*}
\end{proof}

\begin{lem}\label{lem:free pullback}
    Let $\Mfunc: P \to \Vec$ and $\Nfunc: Q \to \Vec$ be persistence modules and let ${f:P \leftrightarrows Q:g}$ be a Galois connection with $\Mfunc \circ g \cong \Nfunc$.
    If $\varphi: \Ffunc \Rightarrow \Mfunc$ is a free cover of $\Mfunc$, then $\varphi \circ g : \Ffunc \circ g \Rightarrow \Nfunc$ is a free cover of $\Nfunc$.
\end{lem}
\begin{proof}
    Since precomposition is additive, Lemma \ref{lem:free pullback 0} implies that $\Ffunc \circ g$ is a free persistence module.
    That $\varphi \circ g$ is surjective follows immediately from the surjectivity of $\varphi$.
\end{proof}

\begin{prop}\label{prop:free presentation}
    For any $S$-constructible persistence module $\Mfunc: P \to \Vec$, there is a free cover $\varphi: \Ffunc \Rightarrow \Mfunc$ where $\Ffunc$ is an $S$-constructible persistence module.
\end{prop}
\begin{proof}
    We proceed by constructing a free cover $\psi:\Gfunc \Rightarrow \Mfunc \big|_S$ and extending it to a free cover $\varphi: \Ffunc \Rightarrow \Mfunc$.
    For any $a \in S$, let $n_a = \dim \Mfunc \big|_S(a)$ and define $\Gfunc_a : S \to \Vec$ as the free persistence module
    \[
        \Gfunc_a = \bigoplus_{i=1}^{n_a} \kk^{\uparrow a}.
    \]
    Let $\psi_a(a): \Gfunc_a(a) \to \Mfunc \big|_S(a)$ be an isomorphism.
    Then $\psi_a(a)$ extends to a natural transformation $\psi_a : \Gfunc_a \to \Mfunc \big|_S$ by
    \[
        \psi_a (b)=
        \begin{cases}
            \Mfunc \big|_S(a \leq b) \circ \psi_a (a) &\text{if } a\leq b\\
            0   &\text{otherwise.}
        \end{cases}
    \]
    We can construct $\Gfunc$ by $\Gfunc = \bigoplus_{a \in S} \Gfunc_a$ and extend $\{ \psi_a \}_{a\in S}$ to a natural transformation $\psi : \Gfunc \to \Mfunc \big|_S$ with the universal property of a direct sum.
    Each $\psi_a$ is surjective at $a$ so $\psi$ is surjective everywhere and therefore, $\psi$ is a free cover of $\Mfunc \big|_S$.

    Since $\Mfunc$ is $S$-constructible, there exists a co-closure operator $c: P \to P$ with image $S$ satisfying $\Mfunc = \Mfunc\circ c$.
    By Proposition \ref{prop: coclosure}, the inclusion $\iota: S \hookrightarrow P$ has a right adjoint $\pi$ satisfying $c = \iota \circ \pi$.
    Define $\Ffunc: P \to \Vec$ as $\Ffunc = \Gfunc \circ \pi$ and $\varphi: \Ffunc \Rightarrow \Mfunc$ as $\varphi = \psi \circ \pi$.
    Now Lemma \ref{lem:free pullback} implies that $\varphi: \Ffunc \Rightarrow \Mfunc$ is a free cover of $\Mfunc$.
\end{proof}

There are often far smaller free covers than the one constructed in the proof of Proposition \ref{prop:free presentation}.
For more details on computing minimal free covers, see \cite{La_Scala_1998,Cox2005-zq,lesnick-free-presentations}.

\begin{ex}
    Constructibility as we've defined it is crucial to guarantee the existence of free covers.
    Consider the non-constructible functor $\Mfunc : \R \to \Vec$ defined by $\Mfunc (a) \cong \kk$ for all $a \in \R$ and $\Mfunc(a \leq b) = \id$ for any $a \leq b$.
    This functor does not have a finite-dimensional free cover.
    Another example is the non-constructible functor $\Nfunc : [0, \infty] \to \Vec$ defined as $\Nfunc(a) \cong \kk$ if $a >0$ and $\Nfunc(0) \cong 0$ with $\Nfunc(a\leq b) = \id_\kk$ for all $a>0$.
    There is no free cover of $\Nfunc$.
\end{ex}

\begin{figure}
    \[
    \begin{tikzcd}
        &\kk^4 & & & 0\\
        &\kk^4 \ar[u, "\id"] &\vphantom{1}\ar[r,Rightarrow, "\varphi"]   &\vphantom{1} & \kk \ar[u]\\
        \kk^3 \ar[ur, "\alpha"] & &\kk \ar[ul, "\beta"'] & \kk^2 \ar[ur, "\theta"] & & \kk \ar[ul, "\id"']
    \end{tikzcd}
    \]
    \caption{
        A free cover of the persistence module $\Mfunc$ from Example \ref{ex:module}.
        The map $\alpha$ is the inclusion into the first three coordinates while $\beta$ is the inclusion into the fourth coordinate.
        See Example \ref{ex:free presentation} for the natural transformation $\varphi: \Ffunc \Rightarrow \Mfunc$.
        }
    \label{fig:free}
\end{figure}

\subsubsection{Birth-Death Functions}

We now define birth-death functions induced by a free cover.

\begin{defn}
    Let $\Mfunc : P \to \Vec$ be a persistence module with a free cover ${\varphi: \Ffunc \Rightarrow \Mfunc}$.
    The \define{birth-death function of $\Mfunc$ induced by $\varphi$} is the map $m_\varphi : \bar P \to \Z$ defined by
    \[
        m_\varphi : (a,b) \mapsto \dim \big( \Ffunc(a) \cap \ker (\varphi_b) \big)
    \]
    where the intersection is given by the pullback
    \begin{center}
        \begin{tikzcd}
            \Ffunc(a)\ar[r, hookrightarrow]	    &\Ffunc(b)\\
            \Ffunc(a)\cap \ker (\varphi_b)\ar[u, dashed]\ar[r, dashed]    &\ker (\varphi_b)\ar[u, hookrightarrow].
        \end{tikzcd}
    \end{center}
\end{defn}

The following lemma relates the nullity function with birth-death functions.

\begin{lem}\label{lem:ker and bd}
    Let $\Mfunc: P \to \Vec$ be a persistence module with a free cover $\varphi: \Ffunc \Rightarrow \Mfunc$.
    Then
    \[
        \mathsf{null}_\Mfunc (a,b) = m_\varphi (a,b) - m_\varphi(a,a)
    \]
    for any $(a,b) \in \bar P$.
\end{lem}
\begin{proof}
    The commutative diagram
    \begin{center}
        \begin{tikzcd}
            \Ffunc(a)\ar[r,hook,"\Ffunc(a\leq b)"]\ar[d,twoheadrightarrow,"\varphi_a"'] &\Ffunc(b)\ar[d,twoheadrightarrow,"\varphi_b"]\\
            \Mfunc(a)\ar[r, "\Mfunc(a \leq b)"] &\Mfunc(b)
        \end{tikzcd}
    \end{center}
    implies that the restriction of $\varphi_a$ onto $\Ffunc(a)\cap \ker(\varphi_b)$, denoted $\phi_{a,b} : \Ffunc(a)\cap \ker(\varphi_b) \to \Mfunc(a)$, has
    \begin{itemize}
        \item $\image (\phi_{a,b}) = \ker \big(\Mfunc(a\leq b)\big)$,
        \item $\ker (\phi_{a,b}) = \ker(\varphi_a)$
    \end{itemize}
    Hence, 
    \[
        \ker \big(\Mfunc(a \leq b)\big) \cong \frac{\Ffunc(a) \cap \ker (\varphi_b)}{\ker (\varphi_a)}.
    \]
    So, 
    \[
        \mathsf{null}_\Mfunc(a,b) = \dim \big( \Ffunc(a) \cap \ker(\varphi_b) \big) - \dim \big( \Ffunc(a) \cap \ker(\varphi_a) \big) = m_\varphi (a,b) - m_\varphi (a,a).
    \]
\end{proof}

\begin{prop}\label{prop: bd-is-const}
    Let $\Mfunc: P \to \Vec$ be a persistence module with a constructible set $S \subseteq P$.
    There exists a free cover $\varphi: \Ffunc \Rightarrow \Mfunc$ such that $m_\varphi$ is a $\bar S$-constructible function.
\end{prop}
\begin{proof}
    Let $c: P \to P$ be the co-closure operator with image $S$.
    By Proposition \ref{prop:free presentation}, we can choose a free cover $\varphi: \Ffunc \Rightarrow \Mfunc$ where $\Ffunc: P \to \Vec$ is an $S$-constructible persistence module.
    Now since $\Ffunc \cong \Ffunc \circ c$, we have
    \begin{align*}
        m_\varphi (a,b) &= \dim \big( \Ffunc(a) \cap \ker(\varphi_b) \big)\\
        &= \dim \Big( \Ffunc \big( c(a) \big) \cap \ker( \varphi_{c(b)} ) \Big)\\
        &=m_\varphi \big( c(a),c(b) \big)
    \end{align*}
    for any $(a,b) \in \bar P$.
    Therefore $m_\varphi = m_\varphi \circ \bar c$ and so $m_\varphi$ is $\bar S$ constructible.
\end{proof}

Next we show that birth-death functions also induce persistence diagrams.

\begin{prop}\label{prop:ker and bd}
    Let $\Mfunc: P \to \Vec$ be a constructible persistence module with a free cover $\varphi: \Ffunc \Rightarrow \Mfunc$.
    Then $\partial m_\varphi$ is a persistence diagram of $\Mfunc$. That is, $\partial m_\varphi$ and $\partial \mathsf{null}_\Mfunc$ agree on off-diagonal intervals.
\end{prop}
\begin{proof}
    For any $a \leq b \in P$, Lemma \ref{lem:ker and bd} shows $m_\varphi (a,b) - m_\varphi (a,a) = \mathsf{null}_\Mfunc (a,b)$.
    Let $n: \bar P \to \Z$ be the function defined by $n(a,b) = m_\varphi(a,a)$.
    Then $m_\varphi - n = \mathsf{null}_\Mfunc$ and so, by the linearity of M\"obius inversion, the claim reduces to showing that $\partial n$ is 0 off the diagonal.
    
    Let $\ell: P \to \Z$ be the function $\ell: a \mapsto m_\varphi(a,a)$.
    Define $f : P \to \bar P$ by $f(a) := (a,a)$ and $g: \bar P \to P$ by $g(b,c) := b$.
    The maps $(f,g)$ are a Galois connection as $(a,a)\leq (b,c)$ if and only if $a \leq b$.
    Since $\ell \circ g = n$, by Theorem \ref{thm:RGCT}
    \[
        \partial n (a,b) = f_\sharp \partial \ell (a,b) = \sum_{c \in f^{-1}(a,b)} \partial \ell(c).
    \]
    Now because $f^{-1}(a,b)$ is empty if $a\neq b$, it follows that $\partial n (a,b)=0$ if $a\neq b$ and so $\partial n$ is 0 outside the diagonal as desired.
\end{proof}

The following corollary is an immediate consequence of Proposition \ref{prop:ker and bd}.

\begin{cor}\label{cor:bd}
    If $\Mfunc : P \to \Vec$ is a constructible persistence module with free covers $\varphi: \Ffunc \Rightarrow \Mfunc$ and $\psi: \Gfunc \Rightarrow \Mfunc$ then $\partial m_\varphi \cong \partial m_\psi$.
\end{cor}

\begin{rmk}[Equivalence with Classical Persistence Diagrams]\label{rmk: establish equivalence}
    Let $K_1 \subseteq K_2 \subseteq \cdots \subseteq K_n$ be a filtration of simplicial complexes.
    Let $\Mfunc$ be the persistence module obtained by applying the degree-$d$ homology functor for a fixed $d\in \mathbb{N}$; that is, $\Mfunc(i) = H_d(K_i)$ for each $i \in \{1, 2,\ldots, n\}$.
    We now show that the persistence diagram of $\Mfunc$, as per our Definition~\ref{def: dgm-of-a-mod}, coincides with the classical degree-$d$ persistence diagram of the filtration.

    The key insight is to use the framework of \cite{editdistance} as a bridge to the classical definition.
    Our strategy is to first show that our persistence diagram, computed via a canonical free cover of $\Mfunc$, aligns with the one defined in \cite{editdistance}.
    We then utilize the fact, established in \cite[Section 9.1]{editdistance}, that their construction recovers the classical persistence diagram.

    Let $\Ffunc$ be the persistence module obtained by applying the $d$-dimensional cycle space functor $Z_d$ to the filtration.
    That is, $\Ffunc(i) = Z_d(K_i)$ for any $i \in \{1,2,\ldots, n\}$.
    Since the indexing poset is finite and totally ordered, $\Ffunc$ is a free module.
    The quotient maps $Z_d(K_i) \to H_d(K_i)$ that define homology assemble into a surjective natural transformation of modules $\varphi: \Ffunc \Rightarrow \Mfunc$, providing a canonical free cover of $\Mfunc$.
    \textit{We caution that this construction of $\Ffunc$ from cycle spaces is only guaranteed to yield a free persistence module because the underlying poset is finite and totally ordered.}

    With this free cover, the birth-death function $m_\varphi$ is the same as the birth-death function from \cite{editdistance}.
    In Section 9.1 of \cite{editdistance}, it was shown that the M\"obius inversion of this birth-death function induces the classical notion of persistence diagrams for finite, totally ordered filtrations.
    Therefore, our definition applied to the module $\Mfunc$ indeed recovers the classical persistence diagram of the filtration.
\end{rmk}

\subsection{Rank Functions}

Persistence diagrams have been historically defined using the rank function rather than birth-death functions or nullity functions \cite{CSEdH,Patel2018}.
The disadvantage of the rank function is that its natural domain is the poset $\hat P$ rather than $\bar P$; see Definition \ref{def: posets-of-ints}.
While a Galois connection between posets $P$ and $Q$ induces a Galois connection between $\bar P$ and $\bar Q$, it fails to induce a Galois connection between $\hat P$ and $\hat Q$.
In fact, order preserving maps from $P$ to $Q$ in general do not induce order preserving maps from $\hat P$ to $\hat Q$.
For this reason, nullity functions and birth-death functions are much more natural to work with in our language.

\begin{defn}
    The \define{rank function} of a persistence module $\Mfunc: P \to \Vec$ is the map $\rank_\Mfunc : \hat P \to \Z$ defined by
    \[
        \rank_\Mfunc : (a,b) \mapsto \dim \Big( \image \big( \Mfunc(a \leq b) \big) \Big).
    \]
\end{defn}

\begin{rmk}\label{rmk:rank/null}
    By introducing a distinguished maximum element $\infty $ to $ P$, the domain of a persistence module $\Mfunc : P \to \Vec $  can be canonically extended to $P\cup\{{\infty}\} $ by defining $\Mfunc (\infty) = 0$.
    In this setting, the nullity function and the rank function carry the same information.
    This follows from the rank-nullity theorem:
    \[
        \dim \Mfunc(a) = \rank_\Mfunc (a,b) + \mathsf{null}_\Mfunc (a,b)
    \]
    for any $a \leq b$.
    Using the fact that $\dim \Mfunc(a) = \rank_\Mfunc (a,a) = \mathsf{null}_\Mfunc (a,\infty)$ we get
    \[
        \rank_\Mfunc (a,b) = \mathsf{null}_\Mfunc (a,\infty) - \mathsf{null}_\Mfunc (a,b) \hspace{1cm} \text{and} \hspace{1cm} \mathsf{null}_\Mfunc (a,b) = \rank_\Mfunc (a,a) - \rank_\Mfunc (a,b).
    \]
\end{rmk}

\begin{table}
    \centering
    \begin{tabular}{|c || c| c| c|| c| c| c|}
        \hline
        Interval & $\mathsf{null}_\Mfunc$ & $m_\varphi$ & $\rank_\Mfunc$ & $\partial \mathsf{null}_\Mfunc$ & $\partial m_\varphi$ & $\partial \rank_\Mfunc$ \\ [0.5ex] 
        \hline\hline
        $(a,a)$ & 0 & 1 & 2 & 0 & 1 & 1\\ 
        \hline
        $(b,b)$ & 0 & 0 & 1 & 0 & 0 & 0\\
        \hline
        $(a,c)$ & 1 & 2 & 1 & 1 & 1 & 1\\
        \hline
        $(b,c)$ & 0 & 0 & 1 & 0 & 0 & 1\\
        \hline
        $(a,d)$ & 2 & 3 & 0 & 1 & 1 & 0\\
        \hline
        $(b,d)$ & 1 & 1 & 0 & 1 & 1 & 0\\
        \hline
        $(c,c)$ & 0 & 3 & 1 & -1 & 1 & -1\\
        \hline
        $(c,d)$ & 1 & 4 & 0 & -1 & -1 & 0\\
        \hline
        $(d,d)$ & 0 & 4 & 0 & -1 & 0 & 0\\
        \hline
    \end{tabular}
    \caption{
        For the persistence module $\Mfunc$ defined in Example \ref{ex:module} with free cover $\varphi: \Ffunc \Rightarrow \Mfunc$ defined in Example \ref{ex:free presentation}, the values of $\mathsf{null}_\Mfunc$, $m_\varphi$, and $\rank_\Mfunc$ are shown, along with the values of their M\"obius inversions.
        }
    \label{tab:diagrams}
\end{table}

\begin{ex}
    Consider the persistence module $\Mfunc$ presented in Example~\ref{ex:module} and illustrated in Figure~\ref{fig:module}.
    Table~\ref{tab:diagrams} lists the values of the nullity function, birth-death function (computed with the free cover $\varphi: \Ffunc \Rightarrow \Mfunc$ defined in Example \ref{ex:free presentation}), and rank function, together with their M\"obius inversions.
    Observe  that $\partial \mathsf{null}_\Mfunc$ and $\partial m_\varphi$ coincide on off-diagonal intervals, as guaranteed by Proposition \ref{prop:ker and bd}.
\end{ex}

\subsection{The Fibered Barcode}

The fibered barcode \cite{landi,RIVET} captures information about a multiparameter persistence module $\Mfunc: [0,\infty)^n \to \Vec$ by restricting $\Mfunc$ to linear slices. Along a fixed line $\ell \subseteq [0,\infty)^n$ with nonnegative slope, the persistence module $\Mfunc |_\ell$ is a 1-parameter persistence module and so any of the usual methods of obtaining a persistence diagram can be implemented. The fibered barcode of $\Mfunc$ is the set of all persistence diagrams over each linear slice with nonnegative slope.

For a fixed line $\ell \subseteq [0,\infty)^n$ with nonnegative slope, let $\iota: [0,\infty) \mono [0,\infty)^n$ be the inclusion of $\ell$ into $[0,\infty)^n$. Since $\ell$ has nonnegative slope, $\iota$ is an order-preserving map. Moreover, $\iota$ has a left adjoint $\pi : [0,\infty)^n \to \ell$ defined by $\pi(x) = \max \{a \in \ell \mid \iota(a) \leq x \}$; see Figure \ref{fig: fibered-barcode}.

\begin{prop}\label{prop: fibered linear time}
    The persistence diagram of the 1-parameter persistence module $\Mfunc |_\ell$ is given by $\pi_\sharp (\partial \mathsf{null}_{\Mfunc})$.
\end{prop}

\begin{proof}
    We have that $\Mfunc|_\ell \cong \Mfunc \circ \iota$ which is a morphism from $\Mfunc$ to $\Mfunc_\ell$ and so, by RGCT (Theorem~\ref{thm:RGCT}), $\pi_\sharp (\partial \mathsf{null}_{\Mfunc}) =  \partial ( \mathsf{null}_{\Mfunc_\ell})$.
\end{proof}

\begin{figure}
    \centering
    \includegraphics[scale=.3]{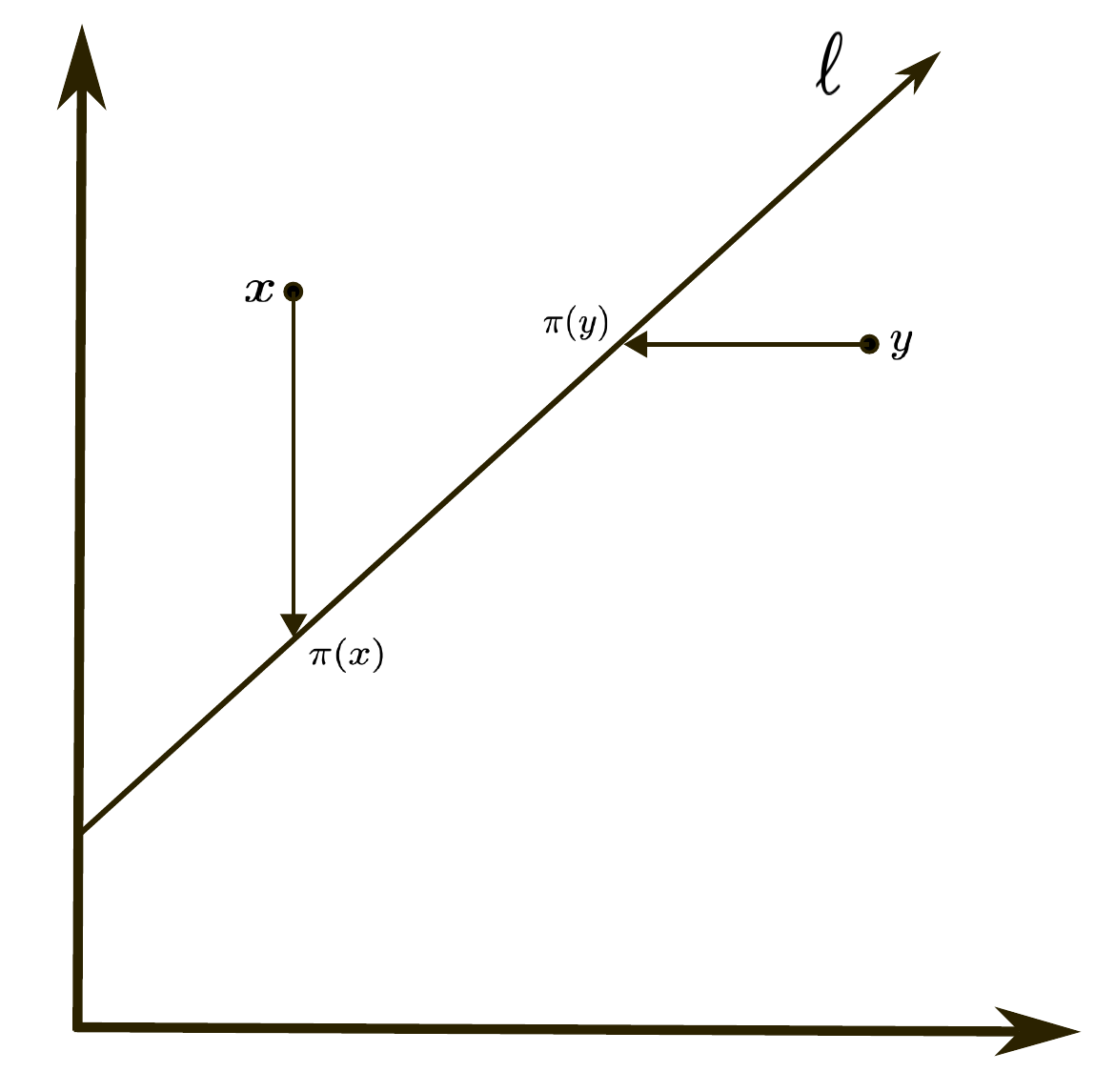}
    \caption{The inclusion $\iota$ of the line $\ell$ into $[0,\infty)^2$ and its left adjoint $\pi$.}
    \label{fig: fibered-barcode}
\end{figure}

\section{Discussion}

The introduction of Galois connections for constructing the category of persistence modules, along with the concept of Galois interleavings, provides a streamlined framework for a proof of classical bottleneck stability.
This language offers a fresh perspective on persistent homology and ties in seamlessly with the existing M\"obius inversion approach via Rota's Galois connection theorem.
However, generalizing these methods to establish a notion of bottleneck stability for multiparameter persistence diagrams presents challenges. Specifically, the presence of negative values in persistence diagrams creates a fundamental obstacle.

In classical settings, matchings are defined only between nonnegative persistence diagrams.
Extending this notion to allow matchings involving negative values results in the bottleneck distance collapsing to zero, as mentioned in Example~\ref{ex: neg-matchings} and Figure~\ref{fig: negative-matchings}.
This issue underscores a critical limitation in extending the bottleneck distance to multiparameter persistent homology.

\newpage

\bibliographystyle{alpha}
\bibliography{ref}{}

\newpage

\appendix

\section{M\"obius Inversion}\label{app: mobius}

Assume that $P$ is a finite poset. Every function $m: P \to \Z$ can be identified with a vector in $\Z^P$ in the natural way.
Under this identification, the M\"obius inversion of $m$ can be calculated with a matrix multiplication as follows.
Let $\zeta_P$ be the $P \times P$-indexed matrix defined by 
    \[(\zeta_P)(a,b) := 
        \begin{cases}
            1   &\text{if } a\leq b\\
            0   &\text{otherwise.}
        \end{cases}
    \]
Multiplication by the $\zeta_P$ matrix has the following effect on functions
    \begin{equation}\label{eq:zeta}
        (m\cdot \zeta_P) (b) = \sum_{a \in P} m(a) \zeta_P(a,b) = \sum_{a\leq b} m(a).
    \end{equation}
The M\"obius inversion formula can now be restated as $m = (\partial m)\cdot \zeta_P$.
The matrix $\zeta_P$ has an inverse matrix $\mu_P$ \cite{rota64}. Therefore the M\"obius inversion of $m$ is 
\begin{align*}
    \partial m &= \partial m \cdot (\zeta_P \cdot \mu_P ) \\
               &= (\partial m \cdot \zeta_P) \cdot \mu_P \\
               &= m \cdot \mu_P.
\end{align*}
This approach implies that M\"obius inversion is a linear map on the space $\Z^P$, a property we will make use of.

\begin{rmk}\label{rmk:rgct}
    The usual statement of Rota's Galois connection theorem \cite{greenemobius,rota64} says that if $f:P \leftrightarrows Q:g$ is a Galois connection, then for any $p \in P$ and $q \in Q$,
    \begin{equation}\label{eqn: old formulation of rgct}
        \sum_{\substack{v \in Q \\ g(v) = p}} \mu_Q (v,q) = \sum_{\substack{u \in P \\ f(u) = q}} \mu_P (p,u).
    \end{equation}
    This is equivalent to Theorem \ref{thm:RGCT} as follows. Fix a $p \in P$ and $q\in Q$. Let $\mathbf{1}_p : P \to \Z$ be the indicator function
    \[
        \mathbf{1}_p (u) :=
        \begin{cases}
            1   &\text{if } u=p\\
            0   &\text{otherwise.}
        \end{cases}
    \]
    Applying Theorem \ref{thm:RGCT} and evaluating both sides at $q$ gives
    \begin{align*}
        \big(\partial (g^\sharp \mathbf{1}_p)\big)(q) &= \big(f_\sharp \partial \mathbf{1}_p \big) (q)\\
        \big( (\mathbf{1}_p \circ g) \mu_Q \big) (q) &= \sum_{u \in f^{-1}(q)} \partial \mathbf{1}_p (u)\\
        \sum_{v \leq q} \mathbf{1}_p (g(v)) \mu_Q (v,q) &= \sum_{u \in f^{-1}(q)} \, \sum_{a \leq u} \mathbf{1}_p (a) \mu_P (a,u)\\
        \sum_{\substack{v \in Q \\ g(v) = p}} \mu_Q (v,q) &= \sum_{\substack{u \in P \\ f(u)=q}} \mu_P (p,u).
    \end{align*}
    This argument shows that Theorem \ref{thm:RGCT} applied to indicator functions implies Equation \ref{eqn: old formulation of rgct}. 

    The argument above also shows that the Equation \ref{eqn: old formulation of rgct} implies Theorem \ref{thm:RGCT} for the indicator functions. This is also noticed in \cite[Lemma 2.2]{sanchezhopf}. Now, let $m: P \to \Z$ be any function. Write $m = \sum_{p\in P} m(p) \mathbf{1}_p$. Then, using the linearity of pullback, pushforward, and M\"obius inversion, we obtain
    \begin{align*}
        \partial_Q \left(g^\sharp m\right) &= \partial_Q \left(g^\sharp \left( \sum_{p\in P} m(p) \mathbf{1}_p \right)\right) \\
        &= \sum_{p\in P} m(p) \partial_Q \left (g^\sharp \left(\mathbf{1}_p\right) \right) \\
        &= \sum_{p\in P} m(p) \left ( f_\sharp \partial \mathbf{1}_p \right ) \\
        &= f_\sharp \partial \left ( \sum_{p\in P} m(p) \mathbf{1}_p \right) \\
        &= f_\sharp \partial m.
    \end{align*}
    
    Hence, we conclude that the Equation \ref{eqn: old formulation of rgct} implies Theorem \ref{thm:RGCT}.
\end{rmk}

\begin{rmk}
    When the poset $P$ is locally finite, the (possibly infinity by infinity) matrix $\zeta_P$ still has the inverse matrix $\mu_P$. However, for a function $m : P \to \Z$, it may no longer hold that $m \cdot \mu_P$ is the M\"obius inversion of $m$. For example, consider $P = (\Z, \leq)$. Then, the matrix
    \[
    \mu_\Z (a,b) = \begin{cases}
        1 &\text{if } a = b \\
        -1 &\text{if } a =b-1 \\
        0 &\text{otherwise}
    \end{cases}
    \]
    satisfies that 
    \[
    \sum_{a\leq x \leq b} \mu_\Z(a,x) \zeta_\Z(x,b) =\sum_{a\leq x \leq b} \zeta_\Z(a,x) \mu_\Z(x,b) = \begin{cases}
        1 &\text{if } a=b \\
        0 &\text{otherwise.}
    \end{cases}
    \]
    That is, $\mu_\Z$ is the inverse of $\zeta_\Z$. Yet, for the function $m : \Z \to \Z$ given by $m(b) = b$, we obtain that $(m \cdot \mu_\Z) (a) = 1$ for all $a\in \Z$. And, it is easy to see that the sum $\sum_{a\leq b \in \Z} (m \cdot \mu_\Z) (a) $ diverges for all $b\in Z$. Thus, $m \cdot \mu_Z$ is not the M\"obius inversion of $m$. In fact, in this case, $m$ does not have a M\"obius inversion.
\end{rmk}

\begin{prop}\label{prop:relation diagrams}
    The persistence diagrams obtained via the nullity function and the rank function are related by the following formula
    \begin{align*}
        \partial_{\hat P} \rank_\Mfunc &= \big(f_\sharp \partial_{\bar P} \mathsf{null}_\Mfunc \big) \zeta_{\bar P} \mu_{\hat P} -\big(\partial_{\bar P} \mathsf{null}_\Mfunc \big) \zeta_{\bar P} \mu_{\hat P}
    \end{align*}
    where $f_\sharp$ is the pushforward of $ f: \bar P \to \bar P$ defined by $ f: (a,b) \mapsto (a,a)$ and $\zeta_{\bullet}$ and $\mu_{\bullet}$ are matrices as defined in Appendix~\ref{app: mobius}.
\end{prop}
\begin{proof}
    Without loss of generality, assume that $P$ has a maximum element $\infty$ such that $\Mfunc(\infty) = 0$.
    We begin by establishing a Galois connection on $\bar P$.
    Let $f , g : \bar P \to \bar P$ be defined by $f : (a,b) \mapsto (a,a)$ and $g : (a,b) \mapsto (a,\infty)$.
    This is a Galois connection since $f(a,b)=(a,a) \leq (c,d)$ if and only if $(a,b) \leq (c,\infty)=g(c,d)$ for any $(a,b), (c,d) \in \bar P$.\\
     Now from Remark \ref{rmk:rank/null}, we have $\rank_\Mfunc = \mathsf{null}_\Mfunc \circ g - \mathsf{null}_\Mfunc$.
     Multiplying both sides on the right by $\mu_{\hat P}$ gives
    \begin{align*}
        \partial_{\hat P} \rank_\Mfunc &= \big( \mathsf{null}_\Mfunc \circ g \big) \mu_{\hat P} - \mathsf{null}_\Mfunc \mu_{\hat P}\\
        &= \big( \mathsf{null}_\Mfunc \circ g \big) \mu_{\bar P} \, \zeta_{\bar P} \, \mu_{\hat P} - \mathsf{null}_\Mfunc \mu_{\bar P} \, \zeta_{\bar P} \, \mu_{\hat P} \\
        &= \big(f_\sharp \partial_{\bar P} \mathsf{null}_\Mfunc \big) \zeta_{\bar P} \mu_{\hat P} - \big( \partial_{\bar P} \mathsf{null}_\Mfunc \big) \zeta_{\bar P} \mu_{\hat P}.
    \end{align*}
\end{proof}

\section{Constructibility}\label{appendix: constructible}

In general, there is no unique set $S$ where a functor is $S$-constructible.
The following two results show that every constructible functor $\Mfunc : P \to \Vec$ has smallest (under the containment order) set $S_{\Mfunc}$ such that the functor is $S_\Mfunc$-constructible.

\begin{lem}\label{lem: int-of-const-sets}
    Suppose a constructible functor $\Mfunc: P \to \Vec$ is both $S_0$-constructible and $S_1$-constructible.
    Then $\Mfunc$ is $S_0 \cap S_1$-constructible.
\end{lem}
\begin{proof}

    Let $c_0,c_1:P \to P$ be the co-closure operators with images $S_0$ and let $S_1$ respectively.
    For $n\in \N$, consider the co-closure operators $C_{2n} := (c_1 \circ c_0) \circ (c_1 \circ c_0) \circ \cdots \circ (c_1 \circ c_0) = (c_1 \circ c_0)^n$, and $C_{2n+1} := c_0 \circ (c_1 \circ c_0)^n$. Since both $c_0$ and $c_1$ are deflationary and their images $S_0$ and $S_1$ are finite, the sequence $\{ C_n \}_{n\in \N}$ must stabilize. That is, there exist $N\in \N$ such that for all $N_0\geq N$ we have that $C_{N_0} = C_N$. Let $C := C_{N}$ and, without loss of generality, assume that $N = 2n$ for some $n\in \N$.
    We will show that $\image(C) = S_0 \cap S_1$ and that $\Mfunc \circ C \cong \Mfunc$.
    Any element in the image of $C$ must be in the image of both $c_0$ and $c_1$ as
    \begin{align*}
        (c_1 \circ c_0) \circ (c_1 \circ c_0) \circ \cdots \circ (c_1 \circ c_0) &= (c_1 \circ c_0)^n = C_{2n} = C_N \\
        &= C_{N+1} \\
        &= c_0 \circ (c_1 \circ c_0)^n
    \end{align*}
    Conversely, if an element $x \in P$ is in the image of both $c_0$ and $c_1$ then by the idempotent property of co-closure operators, $c_0(x)=c_1(x)=x$ and so $x$ must be in the image of $C$.

    Now, to see that $\Mfunc \circ C \cong \Mfunc$, we have the following
    \begin{align*}
        \Mfunc \circ C &= \Mfunc \circ C_N\\
        &= \Mfunc \circ (c_1 \circ c_0) \circ (c_1 \circ c_0)^{n-1} \\
        &= (\Mfunc \circ c_1) \circ \left(c_0 \circ (c_1 \circ c_0)^{n-1}\right) \\
        &= (\Mfunc \circ c_1) \circ C_{N-1} \\
        &\cong \Mfunc  \circ C_{N-1}\\
        &\hspace{2.5cm} \vdots\\
        &\cong \Mfunc \circ C_0 \\
        &= \Mfunc \circ id_P \\
        &= \Mfunc.
    \end{align*}
    This concludes the proof.
\end{proof}

\begin{prop}\label{prop: smallest-const-set}
    Let $\Mfunc: P \to \Vec$ be a constructible functor.
    Then there exists a unique smallest set $S_\Mfunc \subseteq P$ such that $\Mfunc$ is $S_\Mfunc$-constructible.
\end{prop}
\begin{proof}
    By induction, Lemma \ref{lem: int-of-const-sets} extends to say that finite intersections of constructible sets are constructible sets.
    Let $\mathcal S = \{S_i\}_{i\in I}$ be the set of all possible constructible sets for $\Mfunc$ and $S_\Mfunc := \bigcap_{i \in I} S_i$.
    In general, $\mathcal{S}$ will not be finite.
    Choose some $S_0 \in \mathcal S$.
    We have that
    \[
        S_\Mfunc = \bigcap_{i \in I} S_i = \bigcap_{i \in I} \big( S_i \cap S_0 \big)
    \]
    and, since $S_0$ is a finite set, there are only finitely many possibilities for $S_i \cap S_0$.
    Therefore there are only finitely many distinct terms in the intersection and so we can apply Lemma \ref{lem: int-of-const-sets} to deduce that $\Mfunc$ is $S_\Mfunc$-constructible. 
\end{proof}

\end{document}

%% file: commands.tex
\usepackage{nicefrac}
\usepackage{hyperref}
\usepackage{amssymb,mathtools,amsthm, amsmath}
\usepackage[all,cmtip]{xy}
\usepackage{eulervm,xfrac,mathdots}
\usepackage{rotating, setspace}

\newcommand{\blocktheorem}[1]{%
  \csletcs{old#1}{#1}
  \csletcs{endold#1}{end#1}
  \RenewDocumentEnvironment{#1}{o}
    {\par\addvspace{1.5ex}
     \noindent\begin{minipage}{\textwidth}
     \IfNoValueTF{##1}
       {\csuse{old#1}}
       {\csuse{old#1}[##1]}}
    {\csuse{endold#1}
     \end{minipage}
     \par\addvspace{1.5ex}}
}

\newcommand{\R}        	{{\mathbb R}}
\newcommand{\kk}         {\mathbb K}

\newcommand{\Z}        		{{\mathbb Z}}

\newcommand{\mono}		{\hookrightarrow}

\newcommand{\PD}      	    	{{\mathsf{PD}}}
\newcommand{\Pmod}      	{{\mathsf{PMod}}}

\newcommand{\Dgm}          	{{\mathsf{Dgm}}}

\newcommand{\Ffunc}          	{\mathsf{F}}
\newcommand{\Gfunc}          	{\mathsf{G}}

\newcommand{\Mfunc}          	{{\mathsf{M}}}
\newcommand{\Nfunc}          	{{\mathsf{N}}}

\newcommand{\colim}		{\mathsf{colim}}

\newcommand{\image}		{\mathsf{im}}

\newcommand{\nul}       {\mathsf{null}}
\newcommand{\id}			{\mathsf{id}}

\newcommand{\ee}			{\varepsilon}

\newcommand{\op}			{\mathsf{op}}

\newcommand{\down}    {\mathsf{Down}}
\newcommand{\supp}{\text{supp}}

\newcommand\define[1]		{{\bf{#1}}}

\newcommand{\proj}{\mathsf{pr}}

\newtheoremstyle{amit}
{7pt}
{7pt}
{\itshape}
{7pt}
{\bf}
{:}
{.5em}
{}
